\documentclass[12pt,amstex]{amsart}

\usepackage[pdftex]{graphicx}
\usepackage{epstopdf}
\usepackage{amssymb,amsthm,amsmath,epic, ulem, fullpage}
\usepackage{alltt}

\newcommand{\LL}{\mathcal  {L}}
\newcommand{\FF}{\mathcal  {F}}
\newcommand{\MM}{\mathcal  {M}}
\newcommand{\NN}{\mathcal  {N}}
\DeclareFontFamily{OMS}{rsfs}{\skewchar\font'60}
\DeclareFontShape{OMS}{rsfs}{m}{n}{<-5>rsfs5 <5-7>rsfs7 <7->rsfs10 }{}
\DeclareSymbolFont{rsfs}{OMS}{rsfs}{m}{n}
\DeclareSymbolFontAlphabet{\scr}{rsfs}

\newtheorem{theorem}{Theorem}[section]
\newtheorem{lemma}[theorem]{Lemma}
\newtheorem{proposition}[theorem]{Proposition}
\newtheorem{corollary}[theorem]{Corollary}

\theoremstyle{definition}
\newtheorem{definition}[theorem]{Definition}
\newtheorem{example}[theorem]{Example}

\theoremstyle{remark}
\newtheorem{remark}[theorem]{Remark}

\newtheorem{question}[theorem]{Question}

\newcommand{\bQ}{\mathbb{Q}}
\DeclareMathOperator{\Spec}{{Spec}}
\newcommand{\bZ}{\mathbb{Z}}
\newcommand{\tensor}{\otimes}
\DeclareMathOperator{\divisor}{{div}}

\newcommand{\sH}{\scr{H}}
\DeclareMathOperator{\sHom}{{\sH}om}
\newcommand{\blank}{\underline{\hskip 10pt}}
\newcommand{\mydot}{{{\,\begin{picture}(1,1)(-1,-2)\circle*{2}\end{picture}\ }}}
\newcommand{\tld}{\widetilde }
\DeclareMathOperator{\Supp}{{Supp}}
\newcommand{\sL}{\scr{L}}
\newcommand{\bN}{\mathbb{N}}
\newcommand{\bP}{\mathbb{P}}
\DeclareMathOperator{\Proj}{{Proj}}
\newcommand{\sU}{\scr{U}}
\newcommand{\sF}{\scr{F}}
\newcommand{\N}{\mathbb {N}}
\newcommand{\bm}{\mathfrak{m}}
\newcommand{\C}{\mathbb {C}}
\newcommand{\bA}{\mathbb{A}}
\newcommand{\bF}{\mathbb{F}}
\DeclareMathOperator{\Div}{{div}}
\DeclareMathOperator{\Hom}{{Hom}}

\DeclareMathOperator{\Mod}{{mod}}
\newcommand{\cO}{\mathcal{O}}

\input xy
\xyoption{all}

\begin{document}

\title{Globally $F$-regular and log Fano varieties}
\author{Karl Schwede and Karen E.  Smith}
\thanks{The first author was partially supported by a National Science Foundation postdoctoral fellowship and by RTG grant number 0502170 at the University of Michigan.  He also worked on this paper while a research member at MSRI in winter 2009. The second author was partially supported by NSF grant 0500823.}


\address{\vskip -.65cm \noindent Karl E.\ Schwede: \sf Department of Mathematics,
  University of Michigan.  Ann Arbor, Michigan 48109-1109}
\email{kschwede@umich.edu}

\address{\vskip -.65cm \noindent Karen E.\ Smith: \sf Department of Mathematics,
  University of Michigan.  Ann Arbor, Michigan 48109-1109}
\email{kesmith@umich.edu}

\maketitle

\begin{abstract}
We prove that every globally $F$-regular variety  is log Fano.  In other words, if a prime characteristic variety  $X$ is globally $F$-regular,  then it admits an effective $\bQ$-divisor $\Delta$  such that $-K_X - \Delta$ is ample and $(X, \Delta)$ has controlled (Kawamata log terminal, in fact globally $F$-regular) singularities.  A weak form of this result can be viewed as a prime characteristic analog of de Fernex and Hacon's new point of view on Kawamata log terminal singularities in the non-$\bQ$-Gorenstein case.
We also prove a converse statement in characteristic zero: every log Fano variety has globally $F$-regular type. Our techniques apply also to $F$-split varieties, which we show to satisfy a  ``log Calabi-Yau" condition. We also prove a Kawamata-Viehweg vanishing theorem for globally $F$-regular pairs.
 \end{abstract}

\section{Introduction}
\label{SectionIntroduction}
Globally $F$-regular varieties are a special type of Frobenius split variety that have particularly nice properties.  Many well-known Frobenius split varieties are in fact globally $F$-regular,  including toric varieties
\cite{SmithGloballyFRegular} and Schubert varieties  \cite{LauritzenRabenThomsenGlobalFRegularityOfSchuertVarieties}.  Globally $F$-regular varieties are {\it locally\/}  $F$-regular, which implies that they are Cohen-Macaulay and  have pseudo-rational singularities (see \cite{SmithGloballyFRegular} and \cite{SmithFRatImpliesRat}).  On the other hand, demanding $F$-regularity {\it globally\/}   on a projective variety
imposes   strong positivity  properties, and implies  Kodaira-type vanishing results.

Globally $F$-regular varieties were first introduced  in  \cite{SmithGloballyFRegular}, where it was
 observed that  Fano varieties in characteristic zero reduce to globally $F$-regular varieties in characteristic $p \gg 0$. Our purpose in this paper is to consider the  extent to which globally $F$-regular varieties are \emph{essentially equivalent to} Fano varieties. Our main prime characteristic result is the following
 precise way in which globally $F$-regular varieties are ``positive:"
\vskip 6pt
\hskip -12pt
{\bf Theorem 1.1 (Corollary of Theorem \ref{MainExistenceTheorem})}
{\it Let $X$ be a globally $F$-regular variety of prime characteristic.  Then there exists an effective $\bQ$-divisor $\Delta$  such that $(X, \Delta)$ is log Fano. Explicitly, this means that
\begin{itemize}
\item[(i)]  $-K_X - \Delta$ is ample.
\item[(ii)]  $(X, \Delta)$ has  Kawamata log terminal singularities.
\end{itemize}}
\vskip 6pt
Although they are usually considered only in characteristic zero,  Kawamata log terminal singularities are defined in all characteristics; see Remark \ref{kltcharp}.

One immediate consequence (already known to experts) of Theorem 1.1  is that if a smooth projective variety is globally $F$-regular, then  its anticanonical divisor  is big.  However,  this theorem  should be viewed as a much stronger positivity statement about $-K_X$.  Indeed, bigness of a divisor means that its numerical class can be  decomposed as an ample class plus an effective class, but in general, one can not expect to find such a decomposition where the effective class is represented by some effective divisor with controlled singularities. Thus,
the assertion that  $-K_X - \Delta$ is ample  with $(X, \Delta)$ having ``nice'' singularities guarantees  that the big divisor  $-K_X$ is not too far from being ample.

To prove Theorem 1.1,  we develop a theory of  global $F$-regularity for \emph{pairs} $(X, \Delta)$,   generalizing the ``absolute" theory of global $F$-regularity introduced in  \cite{SmithGloballyFRegular}.
 The theory of globally $F$-regular pairs simultaneously generalizes and unifies several points of view, so we expect it will be of interest in its own right;  see Remark \ref{related}.
In fact, Theorem 1.1 is a corollary of a much stronger result, Theorem \ref{MainExistenceTheorem}, which states that if $X$ is globally $F$-regular, then there exists $\Delta$ such that the
 pair $(X, \Delta)$  is  \emph{globally $F$-regular} with $K_X + \Delta$  anti-ample. In particular, the pair  is
\emph{locally $F$-regular}{\footnote{meaning (strongly) $F$-regular on affine patches; Definition \ref{def1} and Remark \ref{related}.}}  with $K_X + \Delta$ $\bQ$-Cartier and therefore the pair is Kawamata log terminal by \cite{HaraWatanabeFRegFPure}.

The converse of Theorem 1.1  is false. However, the failure is due to oddities in small characteristic, and the converse does hold for ``large $p$."    In Section 5, we prove the following:
\vskip 6pt
\hskip -12pt
{\bf Theorem 1.2 (Corollary of Theorem \ref{char0converse})}
{\it Every log Fano projective variety of characteristic zero  has globally $F$-regular type.}
\vskip 6pt
Theorem 1.2 is proved by reducing to a corresponding statement about an affine cone  over the projective variety. Indeed, a projective variety is log Fano if and only if the affine cone over it (for some  appropriately chosen embedding in projective space) has Kawamata log terminal singularities; See Proposition \ref{coneoverlogFano}. Because a corresponding statement for globally $F$-regular singularities also holds (see Proposition \ref{Fregsectionring}), the proof of Theorem 1.2 reduces to a local study of  the singularities at the ``vertex of the cone."

A weak form
of Theorem  \ref{MainExistenceTheorem} can be viewed as a prime characteristic analog of de Fernex and Hacon's new idea for defining Kawamata log terminal singularities for {\it non-} $\bQ$-Gorenstein varieties.  According to the dictionary between the singularities in the minimal model program and $F$-singularities,  Kawamata log terminal singularities  correspond  to  locally
 $F$-regular singularities. Because $F$-regular pairs \emph{do not} require the restriction that $K_X + \Delta$ is $\bQ$-Cartier, it is natural to think of $F$-regular singularities as a prime characteristic analog of Kawamata log terminal singularities \emph{without} this $\bQ$-Cartier assumption. On the other hand, de Fernex and Hacon define a (non-$\bQ$-Gorenstein) variety $X$ to be  Kawamata  log terminal if there exists a $\bQ$-divisor $\Delta$ such that $K_X + \Delta$ is $\bQ$-Cartier and $(X, \Delta)$ is
Kawamata log terminal in the usual sense; see \cite[Proposition 7.2]{DeFernexHacon}. The following local version of Theorem \ref{MainExistenceTheorem} suggests that these two approaches may be equivalent.
\vskip 6pt
\hskip -12pt
{\bf Corollary \ref{CorLocalExistenceTheorem}.}
{\it Let $X$ be a normal affine $F$-finite scheme of prime characteristic $p$.
If $X$ is  $F$-regular,   then there exists an effective $\bQ$-divisor $\Delta$  such that the pair $(X, \Delta)$ is  $F$-regular and $K_X + \Delta$ is $\bQ$-Cartier with index not divisible by $p$.}
\hskip -12pt

Just as log canonical singularities can be viewed as a limit of Kawamata log terminal singularities, we introduce global sharp $F$-splitting,  which can be viewed as a limiting version of global $F$-regularity (in  certain restrictive sense).  This notion is a  global analog of the first author's local notion of sharp $F$-purity \cite{SchwedeSharpTestElements}, and can be thought as a generalization of  Frobenius splitting to the context of pairs.
 Theorem 1.1 (as well as many other results in this paper) has an analog for global sharp $F$-splitting: Theorem \ref{MainExistenceTheorem}  states that if $X$ is Frobenius split, then there exists $\Delta$ such that $(X, \Delta)$ is log Calabi-Yau---that is, $K_X + \Delta$ is $\bQ$-trivial and $(X, \Delta)$ is log canonical.  We also expect a converse to hold which would assert that a projective log Calabi-Yau variety (of characteristic zero) has  $F$-split type; however, this statement remains conjectural and is probably quite difficult to prove; see Remark \ref{RemarkLogCYImpliesFSplit}.

In Section 6, we  explore some of the very strong properties of globally $F$-regular varieties.  For example, in Theorem  \ref{TheoremKawamataViehwegForGloballyFRegular}, we  prove a version of Kawamata-Viehweg vanishing for globally $F$-regular pairs, generalizing the Kodaira-type vanishing theorems from \cite{SmithGloballyFRegular}.   We also observe that many images of global $F$-regular varieties are globally $F$-regular; see Proposition \ref{PropositionMehtaRamanathanVersion}.

Many questions about  globally $F$-regular varieties
and related issues remain open. We conclude the paper with some of these in Section 7.

\hskip -12pt{\it Acknowledgments:}

The authors
would like to thank Nobuo Hara, S\'andor Kov\'acs, Shunsuke Takagi, David E. Anderson and Osamu Fujino for many enjoyable and stimulating discussions.  The first author is particularly grateful for some lively and productive discussions with each of J\'anos Koll\'ar and Holger Brenner,
where the germ of this paper began.
 We would also like to thank the referee, Shunsuke Takagi, Kevin Tucker and Tommaso de Fernex for useful comments on an earlier draft of this paper.

\section{Preliminaries}
\label{SectionPositiveCharacteristicPreliminiaries}

All rings in this paper will be assumed to be commutative with unity, and noetherian.  All schemes will be assumed to be noetherian and separated.  Furthermore, except in Section 5 and unless otherwise specified elsewhere, all rings and schemes will be of positive characteristic $p$.

 \subsection{Frobenius}
 The Frobenius map $X \overset{F}\rightarrow X$ will play a major role in our discussion. The Frobenius map is a morphism of schemes defined as the identity map on the underlying topological space of  $X$,    but the corresponding map of structures sheaves  $\mathcal O_X \rightarrow F_*\mathcal O_X$ sends sections to their $p$-th powers.  For any natural number $e$, the notation $F^e$ denotes the $e$-th iterate of the Frobenius map; in particular, $\mathcal O_X \rightarrow F^e_*\mathcal O_X$ sends sections to their $p^{e}$-th powers.

 \begin{definition}
\label{DefnFFinite} A scheme $X$ of prime characteristic $p$ is $F$-finite if the Frobenius map $\mathcal O_X \rightarrow F_*\cO_X$ gives  $F_*\cO_X$ the structure of a finitely generated $\cO_X$-module, that is, if $F_*\cO_X$ is coherent.   A ring $R$ is called $F$-finite if $\Spec R$ is $F$-finite.
\end{definition}

\begin{remark}
Any scheme essentially of finite type over a field with finite $p$-base (for example, over any perfect field)  is $F$-finite.
\end{remark}

\subsection{Divisors} To fix notation,
we  review the basics of divisors and reflexive sheaves on normal  schemes.   Because normal schemes  are disjoint unions of their irreducible components,  there is no loss of generality working only on normal {\it irreducible\/}  schemes.

Fix a normal  irreducible scheme $X$, with function field $K.$  A \emph{prime divisor} of $X$ is a
 reduced irreducible subscheme of codimension one. A  \emph{Weil divisor} (or simply a ``\emph{divisor}'') is an element of the free abelian group Div($X$) generated by the prime divisors on $X$.
A   \emph{$\bQ$-Divisor} is an element of  Div($X$) $  \tensor_{\bZ} \bQ$.

  A $\bQ$-divisor is called \emph{$\bQ$-Cartier} if there exists an integer $m > 0$ such that $mD$ is a Cartier (or locally principle) divisor.  A ($\bQ$-)divisor  $D$ is  \emph{effective} if all of its coefficients are non-negative. We say that $D' \leq D$ if $D - D'$ is effective.   A $\bQ$-divisor $D$ is \emph{ample} (respectively \emph{big}) if there exists an $m > 0$ such that $m D$ is an ample (respectively big) Cartier divisor.

Associated to any divisor $D$ on $X$ there is a coherent subsheaf $\cO_X(D)$ of $K$, the constant sheaf of rational functions on $X$. Explicitly, for an open set $U \subset X$, we have
$$
\cO_X(D)(U) = \{f \in K \, | \, \divisor_U f + D \cap U \geq 0\}.
$$
Here, recall that  $\divisor_U f $ is the formal sum  over all prime divisors $D_i$ on $U$, where the coefficient of $D_i$ is the order of vanishing of $f$ along $D_i$ (that is, the order of $f$ in the discrete valuation ring $\mathcal O_{X, D_i}$).
A simple but important observation we will use repeatedly is that whenever $D' \leq D$, we have an inclusion $\cO_X(D') \subset \cO_X(D)$,  as subsheaves of K.

\begin{proposition}
\label{PropSectionsAreSameAsEmbeddings}
For a normal irreducible scheme $X$,
the association $$
\Gamma(X, \cO_X(D)) \rightarrow \text{Div} (X)$$
$$
\,\,\,\,\,\,\,\,\,\,\,\,\,\,\,\,\,s \mapsto \divisor \,s + D
$$
defines  a one-to-one correspondence between effective divisors linearly equivalent to $D$ and non-zero global sections of $ \cO_X(D)$ modulo multiplication by units of $ \Gamma(X, \cO_X)$.
\end{proposition}

\subsection{Reflexive sheaves}\label{reflex}
A coherent sheaf $\MM$ on $X$ is said to be {\it reflexive} if the natural map
 $\MM \rightarrow \MM^{\vee\vee}$  to its double dual   (with respect to $\sHom_{\cO_X}(\blank, \cO_X)$) is an isomorphism. The sheaf $\cO_X(D)$  is reflexive for any Weil divisor $D$, but it is not invertible unless $D$ is a Cartier divisor. In fact, every rank one reflexive module   $\MM$ on $X$ is isomorphic to some  $\cO_X(D)$.  In such a case, if $\MM$ has a global section $s$, then the zero set of $s$ (counting
  multiplicities) determines a uniquely defined effective Weil divisor $D$ on $X$ and a corresponding isomorphism of $\MM$ with $\cO_X(D)$ sending $s$ to the element $1$ of the function field $K$.

For the convenience of the reader, we record some useful properties of reflexive modules that we will use without comment.

\begin{proposition}
Let   $\MM$ be a coherent sheaf on  a normal irreducible scheme $X$.  Then:
\begin{itemize}
\item[(1)]  $\MM$ is reflexive if and only if $\MM$ is torsion free and satisfies Serre's S2 condition.
\item[(2)]   If $\NN$ is reflexive, then $\sHom(\MM, \NN)$ is also reflexive. In particular, the sheaf $\MM^{\vee} =
\sHom_{\cO_X}(\MM, \cO_X)$ is reflexive.
\item[(3)]  If $X$ is an $F$-finite scheme of prime characteristic, then $\MM$ is reflexive if and only if $F^e_* \MM$ is reflexive.
\item[(4)]  Suppose $\MM$ is reflexive, and  let $i : U \rightarrow X$ be the inclusion of an open set whose compliment has codimension two.  Then $i_* (\MM|_U) \cong \MM$. Furthermore,  restriction to $U$ induces an equivalence of categories from reflexive coherent sheaves on $X$ to reflexive coherent sheaves on $U$.
\end{itemize}
\end{proposition}

\begin{remark}\label{reflexeasy} We will make especially good use of the fourth statement above. It allows us to treat rank one reflexive sheaves on a normal variety essentially as invertible sheaves: we simply restrict to the regular locus (whose compliment has codimension at least two), perform any operations there, and then extend uniquely to   a
reflexive sheaf on all of $X$.  In terms of Weil divisors, we often argue on the regular locus where they are Cartier, then we take the closure in $X$ to arrive at a statement about Weil divisors on $X$.
\end{remark}

\subsection{The canonical sheaf}\label{canonical}
An important reflexive sheaf is the canonical sheaf $\omega_X$.  When $X$ is of finite type over a
field, $\omega_X$ is the unique reflexive sheaf agreeing with the invertible sheaf  $\wedge^d\Omega_X$ of  top Kahler forms on the smooth locus of $X$, where $d$ is the dimension of $X$.   A \emph {canonical divisor} is any divisor  $K_X$ such that $\omega_X \cong \cO_X(K_X)$.

 More generally, we can work with canonical modules on any integral scheme essentially of finite type over a Gorenstein local ring $R$ (or over any local ring with a dualizing complex). Such a scheme admits a fixed choice of dualizing complex,  namely $\omega_X^{\mydot} =  \eta^{!}R$
  \cite{HartshorneResidues}.  If $X$ is Cohen-Macaulay (and equidimensional),  $\omega_X^{\mydot}$ is concentrated in degree $-d$, and we call this the dualizing sheaf.
So for normal   $X$,  the canonical module $\omega_X$ can be defined as the unique reflexive sheaf which agrees with this dualizing sheaf on the Cohen-Macaulay locus (whose compliment has codimension  at least three in the normal case).

Throughout the rest of the paper, whenever we work with canonical modules, we always assume
they exist in the sense just described. In particular, in prime characteristic, we will assume that all schemes are essentially of finite type over an $F$-finite Gorenstein local ring (or simply over an $F$-finite local ring with a dualizing complex). Importantly, in this case, $F^!\omega_X^{\mydot}$ is quasi-isomorphic to $\omega_X^{\mydot}$.  

\section{Global $F$-regularity and global sharp $F$-splitting of pairs}

We now introduce $F$-splitting and global  $F$-regularity  ``for a  pair" $(X, \Delta)$, and prove an important criterion for them, Theorem \ref{LemmaSplitAlongAndComplementDivisorImpliesFRegular}.  Throughout, we assume
$X$ is  a normal irreducible  scheme{\footnote {In this section, we do not need to assume the existence of dualizing complexes.}} of prime characteristic $p$ and $\Delta$ is an effective  $\bQ$-divisor.
The function field of $X$ is denoted by $K$.

\label{natural} For  any effective divisor $D$ on $X$, we have an inclusion $\cO_X \hookrightarrow \cO_X(D)$   of subsheaves of the function field $K$.  Thus we also have an inclusion $F^e_*\cO_X \hookrightarrow F^e_*\cO_X(D)$. Precomposing with the Frobenius map $\cO_X \rightarrow F^e_*\cO_X$, we have a natural map
$$
\cO_X \rightarrow    F^e_*\cO_X(D)
$$
of coherent $\cO_X$-modules.
It will be the splitting of maps such as these, for particular types of $D$, that provide the definition of $F$-splitting and $F$-regularity.

\begin{definition}\label{def1} Let  $(X, \Delta)$ be a pair,  where $X$ is a normal irreducible $F$-finite  scheme of prime characteristic $p$ and $\Delta$ is an effective  $\bQ$ divisor on $X$.
\begin{itemize}
\item[(i)]
The pair $(X, \Delta)$ is \emph{globally $F$-regular} if, for every effective divisor $D$, there exists some $e > 0$  such that the natural map
 $\cO_X \rightarrow F^e_* \cO_X( \lceil (p^e - 1) \Delta \rceil + D) $  splits (in the category of  $\cO_X$-modules).
 \item[(ii)]
 Similarly, the pair $(X, \Delta)$ is \emph{globally sharply $F$-split} if there exists an $e > 0$
for which the natural map  $ \cO_X \rightarrow F^e_* \cO_X(\lceil (p^e - 1) \Delta \rceil)$
splits.
\item[(iii)]  The pair $(X, \Delta)$ is \emph{locally $F$-regular} if  $X$ can be covered by open sets  $U_i$ such that each $(U_i, \Delta|_{U_i})$ is
 globally  $F$-regular.  {\emph{Local  sharp $F$-splitting}} for a pair is defined similarly.
   \end{itemize}
   \end{definition}

    In general,
 the  global conditions are {\it much stronger\/} than their local counterparts. For example, every smooth projective curve  is locally $F$-regular, but only the genus zero curves are globally $F$-regular; see Example \ref{curve}.  On the other hand, for an affine scheme $X$,  the pair $(X, \Delta)$ is globally  sharply $F$-split (respectively, globally  $F$-regular) if and only if it is    locally sharply $F$-split (respectively, locally
     $F$-regular).

\begin{remark}\label{related} Global $F$-regularity and global sharp $F$-splitting of pairs generalizes and unifies several points of view.  If $X$  is projective and  $\Delta= 0$,    global sharp $F$-splitting for the pair $(X, 0)$ recovers the usual notion of $F$-splitting for $X$ due to Mehta and Ramanathan  \cite{MehtaRamanathanFrobeniusSplittingAndCohomologyVanishing}, while  global $F$-regularity for $(X, 0)$  recovers Smith's notion of global $F$-regularity for $X$ \cite[Thm 3.10]{SmithGloballyFRegular}. (The latter is  itself a stable version of Ramanan and Ramanathan's ``Frobenius splitting along a divisor" \cite{RamananRamanathanProjectiveNormality};    See also
\cite{HaraWatanabeYoshidaReesAlgebrasOfFRegularType}).     At the opposite extreme, for an affine scheme $X =  \Spec R$,  sharp $F$-splitting for $(X, \Delta)$  is simply Schwede's notion of sharp $F$-purity for $(R, \Delta)$, itself a generalization to pairs of Hochster and Robert's $F$-purity for $R$, see \cite{HochsterRobertsRingsOfInvariants} and \cite{HochsterRobertsFrobeniusLocalCohomology} and a slight variant of the notion of (non-sharp) $F$-purity for pairs found in \cite{HaraWatanabeFRegFPure}.  Similarly, our notion of $F$-regularity in the affine setting is simply Hara and Watanabe's strong $F$-regularity for the pair $(R, \Delta)$, which is itself a generalization to pairs  of Hochster and Huneke's  notion  of strong
$F$-regularity for $R$; see \cite{HaraWatanabeFRegFPure} and \cite{HochsterHunekeTightClosureAndStrongFRegularity}.
     \end{remark}

       \begin{remark}\label{opencondition}
       The locus of points of $X$ where a pair $(X, \Delta)$ is locally  $F$-regular (or locally sharply $F$-split) is a non-empty open set of $X$. Indeed, the locus of points  where each  $\cO_X \rightarrow F^e_* \cO_X( \lceil (p^e - 1) \Delta \rceil + D) $  fails to split is closed, and the intersection of arbitrarily many closed sets (as we range over all $e$ and all effective $D$) is closed.  On the other hand, restricting to an open affine set where $X$ is non-singular and $\Delta$ is empty,  we obviously have a splitting.
            \end{remark}

    \begin{remark} \label{concrete}
For any effective divisor $A$,  the sheaf  $F^e_* \cO_X(A)$ is a  subsheaf of the sheaf $F^e_*K$, where $K$ is the function field of $X$.
If we like, we can identify  $F^e_*K$ with the field $K^{1/p^e}$, and view both $\cO_X$ and
$F^e_* \cO_X(A)$ as subsheaves of  $K^{1/p^e}$.  In this case,
 finding a splitting as specified by Definition \ref{def1} is simply a matter of finding an
$\cO_X$-linear map $F^e_* \cO_X(A) \rightarrow \cO_X$ which sends the element 1 to 1, where here $A$ stands for either
$\lceil (p^e - 1) \Delta \rceil$ or
  $\lceil (p^e - 1) \Delta \rceil + D$.
\end{remark}

The next lemma is extremely simple, but we will use it (and the technique of its proof) repeatedly.
\begin{lemma}\label{simple}
If $(X, \Delta)$ is globally $F$-regular, then $(X, \Delta')$ is globally $F$-regular for any $\Delta' \leq \Delta$. The corresponding statement for globally sharply $F$-split pairs also holds.
\end{lemma}

\begin{proof}
This follows easily from the following simple observation: If a map of coherent sheaves  $\LL \overset{g}\rightarrow \FF$ on a scheme $X$ splits, then there is also a splitting for any map $\LL \overset{h}\rightarrow \MM$  through which $g$ factors.
Indeed, factor $g$ as $\LL \overset{h}\rightarrow \MM \overset{h'} \rightarrow \FF$. Then if $s: \FF \rightarrow \LL$ splits $g$, it is clear that the composition $s \circ h'$ splits $h$.
Now we simply observe that if $\Delta' \leq \Delta$, we have a factorization
$$
\cO_X \rightarrow F^e_* \cO_X( \lceil (p^e - 1) \Delta' \rceil + D) \hookrightarrow  F^e_* \cO_X( \lceil (p^e - 1) \Delta \rceil + D),$$ so the result follows.
\end{proof}

\begin{remark}  Lemma \ref{simple} implies that if $(X, \Delta)$ is globally $F$-regular for some $\Delta$, then also $(X, 0)$ is globally $F$-regular, which is to say the scheme $X$ is globally  $F$-regular in the original sense of \cite{SmithGloballyFRegular}.
\end{remark}

\begin{remark}  If $X$ is quasi-projective over an $F$-finite ring, then Definition \ref{def1} (i) of global $F$-regularity for  $(X, \Delta)$ can be stated with {\it very ample} effective  $D$.
To see this, note that if $X$ is quasi-projective, then for every effective divisor $D$ there exists a very ample effective divisor $A$ such that $D \leq A$, so the splitting for $A$ implies the splitting for $D$,
 as in the proof of  Lemma \ref{simple}.
 A more interesting fact is that it is often enough to check the splitting condition \emph{for one choice} of ample divisor $D$; see Theorem \ref{LemmaSplitAlongAndComplementDivisorImpliesFRegular}.
\end{remark}

\begin{proposition}
Let $X$ be a normal $F$-finite scheme of prime characteristic $p$, with  effective $\bQ$-divisor $\Delta$.
\begin{itemize}
\item[(a)]
The pair
 $(X, \Delta)$ is globally $F$-regular if and only if for all effective divisors $D$ and all $e \gg 0$, the map
\[
\cO_X \rightarrow F^{e}_* \cO_X(\lceil (p^{e} - 1) \Delta \rceil + D)
\]
splits.
\item[(b)]  The pair
 $(X, \Delta)$ is globally sharply $F$-split if and only if there exists an $e_0 > 0$ such that
\[
\cO_X \rightarrow F^{ne_0}_* \cO_X(\lceil (p^{ne_0} - 1) \Delta \rceil)
\]
splits for all $n > 0$.

\end{itemize}
\end{proposition}
\begin{proof}
One implication of (a) is trivial, so we assume that $(X, \Delta)$ is globally $F$-regular. Now given any effective divisor $D$, we wish find a natural number $e_0$ such that
\[
\cO_X \rightarrow F^{e}_* \cO_X(\lceil (p^{e} - 1) \Delta \rceil + D)
\]
splits for all $e>e_0$.
 Let $E$ be an effective divisor such that $$E + \lceil (p^e - 1)\Delta \rceil \geq \lceil p^e \Delta \rceil$$ for all $e > 0$.  Then, because $(X, \Delta)$ is globally $F$-regular and $D+E$ is effective, we can find $e_0 $ such that the map
\[
\cO_X \rightarrow F^{e_0}_* \cO_X(\lceil (p^{e_0} - 1) \Delta \rceil + D + E)
\]
splits.   Since this map factors through  $F^{e_0}_* \cO_X(\lceil p^{e_0} \Delta \rceil + D )$ (see the
 proof of Lemma \ref{simple}), this implies a splitting of
 \begin{equation}
 \label{1}
\cO_X \rightarrow F^{e_0}_* \cO_X(\lceil p^{e_0} \Delta \rceil + D ).
\end{equation}

 Now because $X$ itself is $F$-split, we know that $\cO_X \rightarrow F^{f}_*\cO_X$ splits for all $f > 0$.
 Tensoring{\footnote{Strictly speaking, we tensor after first restricting to the nonsingular locus of $X$,  where all sheaves are locally free, then extend uniquely to maps of reflexive sheaves on all of $X$. Equivalently, we can double-dualize after tensoring to recover the reflexivity. These are equivalent since a reflexive sheaf on a normal scheme is determined by its restriction to any open set whose compliment has codimension at least two. We will use this technique for handling reflexive sheaves throughout with little comment; See Subsection \ref{reflex}.}}
 with $ \cO_X(\lceil p^{e_0} \Delta \rceil + D )$,  we have a splitting of
 $$
  \cO_X(\lceil p^{e_0} \Delta \rceil + D ) \rightarrow  F^{f}_*\cO_X \otimes
    \cO_X(\lceil p^{e_0} \Delta \rceil + D ) \cong   F^f_*\cO_X(p^f\lceil p^{e_0} \Delta \rceil + p^f D )
 $$
  for all $f > 0$ (with a natural isomorphism coming from the  projection formula for $F^f$).
 Applying the functor $F^{e_0}_*$, and then  precomposing with the natural map from line (\ref{1}) above, we have a splitting of the map
 \smallskip
\[
\cO_X \rightarrow F^{e_0}_* \cO_X(\lceil (p^{e_0} ) \Delta \rceil + D) \rightarrow
 F^{e_0+f}_*\cO_X(p^f \lceil (p^{e_0} ) \Delta \rceil + p^fD),
\]
\smallskip
 for all $f > 0.$
Finally, because $p^f \lceil p^{e_0} \Delta \rceil + p^f D \geq \lceil (p^{e_0 + f} -1)\Delta \rceil + D$ for all $f$, we conclude that
\[
\cO_X \rightarrow F^{e_0 + f}_* \cO_X(\lceil (p^{e_0 + f} -1)\Delta \rceil + D)
\]
also splits, again  as in the proof of Lemma \ref{simple}.
 The proof  of (a) is complete.

 The proof of  (b) follows as in \cite[Proposition 3.3]{SchwedeSharpTestElements} and \cite{SchwedeCentersOfFPurity}.
\end{proof}

\smallskip
We now establish a useful criterion for  global $F$-regularity, generalizing well-known
results for the local case \cite[Theorem 3.3]{HochsterHunekeTightClosureAndStrongFRegularity} and the
 ``boundary-free" case \cite[Theorem 3.10]{SmithGloballyFRegular}.
\begin{theorem}
\label{LemmaSplitAlongAndComplementDivisorImpliesFRegular}
The pair  $(X, \Delta)$ is globally $F$-regular if (and only if)   there exists some effective divisor $C$ on $X$  satisfying the following two properties:
\begin{itemize}
\item[(i)]  There exists an $e > 0$ such that the natural  map
\[
\xymatrix{
\cO_{X}  \rightarrow   F^e_* \cO_{X}(\lceil (p^e - 1)\Delta + C \rceil)
}
\]
splits.
\item[(ii)]  The pair $(X \setminus C, \Delta|_{X \setminus C})$ globally $F$-regular (for example, affine and locally $F$-regular).
\end{itemize}
\end{theorem}

\smallskip
\begin{corollary}\label{alternate}
The pair $(X, \Delta)$ is globally $F$-regular if and only if there exists an effective $C$ such that
$(X, \Delta)$ is globally $F$-regular on the open set complimentary to $C$ and $(X, \Delta+ \epsilon C)$
is globally sharply $F$-split for sufficiently small $\epsilon$.
\end{corollary}

\begin{remark} In fact, it turns out that  if $(X, \Delta)$ is globally $F$-regular, then for any effective $C$, we have that $(X, \Delta + \epsilon C)$ is also globally $F$-regular for sufficiently small $\epsilon$.
See Corollary \ref{CorGlobalFRegularityIsOpen}.
\end{remark}
\begin{proof}[Proof of Corollary \ref{alternate}]
Condition (i) of Theorem \ref{LemmaSplitAlongAndComplementDivisorImpliesFRegular} obviously implies global sharp $F$-splitting for the pair $(X, \Delta + \frac{1}{p^e-1}C)$, and this in turn implies   global sharp $F$-splitting for the pair
 $(X, \Delta +\epsilon C)$  for any $\epsilon \leq \frac{1}{p^e-1}$, by Lemma \ref{simple}.
 Conversely, if $(X, \Delta +\epsilon C)$ is globally sharply $F$-split, then by Lemma \ref{simple}, we may assume $\epsilon = {1 \over p^g - 1}$ for some sufficiently large $g$.  Now take $n>0$ such that the map
\[
 \cO_X \rightarrow F^{ng}_* \cO_X( \lceil (p^{ng}-1)( \Delta + \frac{1}{p^{g} -1} C )\rceil) = F^{ng}_* \cO_X( \lceil (p^{ng}-1) \Delta  \rceil + \left({p^{ng} - 1 \over p^g- 1}\right) C)
\]
splits.  Since  $C \leq \left({p^{ng} - 1 \over p^g- 1}\right) C$,  we also have a splitting for the map
\[
 \cO_X \rightarrow F^{ng}_* \cO_X( \lceil (p^{ng}-1) \Delta  \rceil + C).
\]
Then setting $e = ng$ we arrive at condition (i) of Theorem
\ref{LemmaSplitAlongAndComplementDivisorImpliesFRegular}, and the proof of the corollary is complete.
\end{proof}

\begin{proof}[Proof of Theorem \ref{LemmaSplitAlongAndComplementDivisorImpliesFRegular}]
 Let $X_C$ denote the open set complimentary to $C$.  Now fix any effective divisor $C'$ on $X$.
By hypothesis (ii), we can find $e'$ and an $\cO_X$-module homomorphism $\phi : F^{e'}_* \cO_{X_C}(\lceil (p^{e'} - 1)\Delta|_{X_C} + C'|_{X_C} \rceil) \rightarrow \cO_{X_C}$ that sends $1$ to $1$.   In other words, $\phi$ is a section of the reflexive sheaf
$$
 \sHom_{\cO_X}(F^{e'}_* \cO_{X}(\lceil (p^{e'} - 1)\Delta + C' \rceil), \cO_{X})
$$
over the open set $X_C$. Thus on the non-singular locus $U$ of $X$ (really, we need the Cartier locus of $C$), we can choose $m_0 > 0$ so that $\phi|_U$ is the restriction of a global section $\phi_m$ of
\begin{equation}
\label{EqnAGoodPsi}
\begin{split}
\sHom_{\cO_U}(F^{e'}_* \cO_{U}(\lceil (p^{e'} - 1)\Delta + C' \rceil), \cO_{U}) \tensor \cO_U(mC) \\
\cong \sHom_{\cO_U}(F^{e'}_* \cO_{U}(\lceil (p^{e'} - 1)\Delta + C' \rceil), \cO_{U}(mC))
\end{split}
\end{equation} over $U$,
for all $m \geq m_0$; see \cite[Chapter II, Lemma 5.14(b)]{Hartshorne}.  Note that $\phi_m$ still sends $1$ to $1$.  Now, since the involved sheaves are reflexive, this section extends uniquely to a global section of $\sHom_{\cO_X}(F^{e'}_* \cO_{X}(\lceil (p^{e'} - 1)\Delta + C' \rceil), \cO_{X}(mC))$, also denoted $\phi_m$ over the whole of $X$.

Consider an $m$ of the form $m = p^{(n-1)e} + \ldots p^e + 1$, where $e$ is the number guaranteed by hypothesis (i).
Tensoring{\footnote{See Remark \ref{reflexeasy} and footnote (3).}}  the map $\phi_m$ from Equation (\ref{EqnAGoodPsi}) with  $\cO_X(\lceil (p^{ne} - 1)\Delta \rceil)$, we have an induced map
\[
F^{e'}_* \cO_{X}(\lceil (p^{e'} - 1) \Delta \rceil + C' + p^{e'}\lceil (p^{ne} - 1)\Delta \rceil)  \rightarrow \cO_X(\lceil (p^{ne} - 1)\Delta + mC\rceil).
\]
Now, as in Lemma \ref{simple}, it follows that there is a map
\[
\psi : F^{e'}_* \cO_{X}(\lceil (p^{ne + e'} - 1) \Delta + C' \rceil)  \rightarrow \cO_X(\lceil (p^{ne} - 1)\Delta + mC \rceil)
\]
which sends $1$ to $1$.

By composing the splitting from hypothesis (i) with itself $(n-1)$-times and after twisting appropriately (compare with \cite[Proof of Lemma 2.5]{TakagiInterpretationOfMultiplierIdeals} and \cite{SchwedeFAdjunction}), we obtain a map
\[
\theta : F^{ne}_* \cO_{X}(\lceil (p^{ne} - 1) \Delta + (p^{(n-1)e} + \dots + p^{e} + 1)C\rceil) = F^{ne}_* \cO_X(\lceil (p^{ne} - 1) \Delta + mC \rceil) \rightarrow \cO_{X}
\] which sends $1$ to $1$.

Combining the maps $\theta$ and $\psi$, we obtain a composition
\[
\xymatrix{
F^{ne + e'}_* \cO_{X}(\lceil (p^{ne + e'} - 1) \Delta + C'\rceil) \ar[r]^-{F^{ne}_* (\psi_{\Delta})} & F^{ne}_* \cO_X(\lceil (p^{ne} - 1)\Delta + mC \rceil) \ar[r]^-{\theta} & \cO_X
}
\]
which sends $1$ to $1$ as desired. The proof is complete.
\end{proof}

\medskip
The next result is important, because it lets us ``deform" the divisor $\Delta$  in a globally $F$-regular pair so as to assume that the coefficients of $\Delta$ do not have denominators divisible by the characteristic $p$.
\begin{proposition}
\label{CorCanReplaceBByNice}
Let $(X, \Delta)$ be a  globally $F$-regular pair.  There exists an effective $\bQ$-divisor $\tld \Delta$ on $X$ such that:
\begin{itemize}
\item[(a)]  $\tld \Delta \geq \Delta$.
\item[(b)]  $(X, \tld \Delta)$ is globally $F$-regular.
\item[(c)]  No coefficient of $\tld \Delta$ has a denominator divisible by $p$ (in other words, there exists an integer $e_0 > 0$ such that $(p^e - 1)\tld \Delta$ is integral for all $e = ne_0$, $n > 0$ an integer).
\end{itemize}
The analogous statement for globally sharply $F$-split pairs also holds.
\end{proposition}
\begin{proof}
We first do the globally sharply $F$-split case.  Since $(X, \Delta)$ is globally sharply $F$-split, for some $e > 0$, the map
\[
\cO_X \rightarrow F^e_* \cO_X(\lceil (p^e - 1)\Delta \rceil)
\]
splits.  Set $\tld \Delta = {1 \over p^e - 1}\lceil (p^e - 1)\Delta \rceil$.  It is easy to see that $\tld \Delta$ satisfies properties (a), (b) and (c) above (or rather, instead of (b), we see that $(X, \tld \Delta)$ is globally sharply $F$-split).

Now suppose that $(X, \Delta)$ is globally $F$-regular.  Choose  an effective divisor $C$ whose support contains the support of $\Delta$.  Since $(X, \Delta)$ is globally $F$-regular, there exists an $e > 0$ so that the natural map
\[
\cO_{X} \rightarrow F^e_* \cO_{X}(\lceil (p^e - 1)\Delta + C \rceil)
\]
splits.  Again, set $\tld \Delta = {1 \over p^e - 1}\lceil (p^e - 1)\Delta \rceil$ and then use Theorem \ref{LemmaSplitAlongAndComplementDivisorImpliesFRegular}.
\end{proof}

\section{Positivity of globally $F$-regular varieties}

In this section, we prove our main prime characteristic result:  globally $F$-regular varieties are log Fano. We also establish what can be  loosely called  a ``limiting statement:"  globally sharply $F$-split varieties are log Calabi-Yau.  The converse statements are false, however some closely related statements are true in characteristic zero;  see Section \ref{SectionLogFanoImpliesFType}.
The main point is Theorem \ref{MainExistenceTheorem} below, which describes the geometry of globally F-regular pairs. Its proof occupies most of the section.


We first recall  the definitions of  log Fano and log Calabi-Yau.

\begin{definition}
Let  $X$ be a  normal irreducible variety, $\Delta$ an effective $\bQ$-divisor on $X$.
\begin{itemize}
\item[(i)] The pair $(X, \Delta)$ is
\emph{log Fano} if $-(K_X + \Delta)$ is ample and the pair $(X, \Delta)$ is Kawamata log terminal.
\item[(ii)]
Similarly, the pair $(X, \Delta)$ is
 \emph{log Calabi-Yau} if $K_X + \Delta$ is $\bQ$-linearly equivalent to zero and the pair $(X, \Delta)$ is log canonical.
 \end{itemize}
 We say that a normal variety is log Fano (or log Calabi-Yau) if there exists a $\bQ$-divisor $\Delta$ on $X$ such that $(X, \Delta)$ is log Fano (or log Calabi-Yau).

 \end{definition}
\smallskip
\begin{remark}\label{kltcharp}
Kawamata log terminal singularities and log canonical singularities are typically considered only in characteristic zero. However, these notions make sense in any characteristic by considering discrepancies along \emph{every} divisor  over $X$. Specifically, a pair $(X, \Delta)$, where $X$ is a normal variety and $\Delta$ is an effective $\bQ$-divisor, is {\it Kawamata log terminal (respectively log canonical)\/} if $K_X + \Delta$ is $\bQ$-Cartier, and for all proper birational maps $Y \overset{\pi}\rightarrow X$, writing $\pi^*(K_X + \Delta) = K_Y + \Delta'$ with $\pi_*K_Y = K_X$, we have that $\lfloor \Delta' \rfloor \leq 0$ (respectively, the coefficients of $\Delta'$ are at most one). In characteristic zero, one need check this condition only for one $\pi$, namely, for any log resolution of $(X, \Delta)$.  We often use the short-hand notation \emph{klt} (respectively \emph{lc}) instead of Kawamata log terminal (respectively, log canonical).
See  \cite[Definition 2.34]{KollarMori}. \end{remark}

We now state our main result:

\begin{theorem}\label{MainExistenceTheorem} Let $X$ be a normal scheme quasiprojective over an $F$-finite local ring with a dualizing complex and suppose that $B$ is an effective $\bQ$-divisor on $X$.
\begin{itemize}
\item[(i)]
 If the pair $(X, B)$ is globally $F$-regular, then there is an effective $\bQ$-divisor $\Delta$ such that $(X, B + \Delta)$ is globally $F$-regular with $K_X + B +\Delta$ anti-ample.
 \item[(ii)]
 Similarly, if $(X, B)$  is  globally sharply  $F$-split, then there exists an effective $\bQ$-divisor $\Delta$ such that
 $(X, B + \Delta)$ is globally sharply $F$-split with $K_X + B + \Delta$ $\bQ$-trivial.
\end{itemize}
In both (i) and (ii), the denominators of the coefficients of $B + \Delta$ can be assumed not divisible by the characteristic $p$.
\end{theorem}

Given this theorem, Theorem 1.1  from  the  introduction follows easily. Indeed, since globally $F$-regular varieties are locally $F$-regular (and respectively, globally sharply $F$-split singularities are locally sharply $F$-split), Theorem 1.1  follows immediately from the following well-known theorem about $F$-singularities:

\begin{theorem}\cite[Theorem 3.3]{HaraWatanabeFRegFPure}
\label{TheoremHaraWatanabePairsImplication}
Let $(X, \Delta)$ be a pair, where $X$ is normal and essentially of finite type over a $F$-finite field, and $\Delta$ is  an effective $\bQ$-divisor such that $K_X + \Delta$ is $\bQ$-Cartier.
\begin{itemize}
\item[(i)]  If $(X, \Delta)$ is locally $F$-regular, then $(X, \Delta)$ is Kawamata log terminal.
\item[(ii)] If $(X, \Delta)$ is locally sharply $F$-split, then $(X, \Delta)$ is log canonical. \end{itemize}
\end{theorem}

Before beginning the proof of Theorem \ref{MainExistenceTheorem}, we point out
a simple corollary.
\begin{corollary}\label{big}
The anticanonical divisor on a globally $F$-regular $\bQ$-Gorenstein  projective variety is big.
\end{corollary}
\begin{proof}
Theorem \ref{MainExistenceTheorem} implies that there exists an effective $\bQ$-divisor $\Delta$ such that $-K_X - \Delta$ is ample.  Therefore $-K_X$ is $\bQ$-linearly equivalent to an ample divisor plus an effective divisor,   which implies $-K_X$ is big by \cite[Cor 2.2.7]{LazarsfeldPositivity1}.
\end{proof}

\begin{remark}
Corollary \ref{big}  was known to experts (for example, see \cite[Remark 1.3]{HaraWatanabeYoshidaReesAlgebrasOfFRegularType}). One way to deduce it directly follows:  if $X$ is globally $F$-regular and  $H$ is ample, then there exists an $e$ so that $\cO_X((1 - p^e)K_X - H)$ has a non-zero global section corresponding to an effective divisor $D$.  Then $-K_X \sim_{\bQ} {1 \over p^e - 1}(H + D)$, which represents $-K_X$ as  ``ample plus effective." \end{remark}

\begin{example}\label{curve} A smooth projective curve is globally $F$-regular if and only if it has genus zero, as higher genus curves  have  canonical divisors with non-negative degree.  A smooth projective curve of genus higher than one is never $F$-split, whereas a curve of genus one is (globally sharply) $F$-split if and only if it is ordinary (non-supersingular); see \cite[Example 4.3]{SmithVanishingSingularitiesAndEffectiveBounds}.
 \end{example}

\begin{example} Ruled surfaces supply many examples of smooth surfaces with big anti-canonical divisor, where the ``difference" between $-K_X$ and the ample cone is too wide  to be breeched by a mildly singular effective  $\Delta$.    Indeed, a ruled surface is globally $F$-regular if and only if the base curve is $\bP^1$. We prove this in Example \ref{ruled}.
  \end{example}

The rest of this section is committed to the proof of  Theorem \ref{MainExistenceTheorem}. We begin with following crucial lemma,  interesting in its own right.

  \begin{lemma}
\label{CorollaryCombinationOfFRegularAndFPureImpliesFRegular}
Consider two pairs $(X, B)$ and $(X, D)$ on a normal $F$-finite scheme $X$.
\begin{itemize}
\item[(i)] If both pairs are globally sharply $F$-split,  then there exist   positive rational
numbers $\epsilon $ arbitrarily close to zero  such that the pair $(X, \epsilon B + (1 - \epsilon) D)$ is  globally sharply $F$-split.
\item[(ii)]
If  $(X, B)$ is globally $F$-regular  and $(X, D)$ is globally sharply $F$-split, then there exist   positive rational
numbers $\epsilon $ arbitrarily close to zero  such that the pair $(X, \epsilon B + (1 - \epsilon) D)$ is  globally $F$-regular.
\item[(iii)]
In particular, if $(X, B) $ is globally $F$-regular and  $(X, B + \Delta)$ is globally sharply $F$-split,  then $(X, B + \delta \Delta)$ is  globally $F$-regular for all rational $0 < \delta < 1$.
\end{itemize} In (i) and (ii), the number $\epsilon$ can be assumed to have denominator not divisible by $p$.
\end{lemma}

\begin{proof}[Proof of Lemma \ref{CorollaryCombinationOfFRegularAndFPureImpliesFRegular}]
First note that (iii) follows from (ii) by taking $D$ to be $(B +  \Delta)$.  Since $(1-\epsilon)$ can be taken to be arbitrarily close to 1, we can choose it to exceed any given  $\delta < 1$. Hence, the pair $(X, B + \delta \Delta)$ is globally $F$-regular for all positive $\delta <1$, by Lemma \ref{simple}.

For (i), we  prove that we can take $\epsilon$ to be any rational number of the form
\begin{equation}\label{eps}
\epsilon = {p^e - 1 \over p^{(e + f)} - 1}
\end{equation}
where $e$ and $f$ are sufficiently large and divisible (but independent) integers.
Take  $e$  large and divisible enough so there exists a map $\phi : F^e_* \cO_X(\lceil (p^e -1) B\rceil) \rightarrow \cO_X$ which splits the map $\cO_X \rightarrow F^e_* \cO_X(\lceil (p^e -1) B\rceil)$.  Likewise,  take  $f$ large and divisible enough so there exists a map $\psi : F^f_* \cO_X(\lceil (p^f -1) D \rceil) \rightarrow \cO_X$ which splits the map $\cO_X \rightarrow F^f_* \cO_X(\lceil (p^f -1) D \rceil)$.

Consider the splitting
\[
 \xymatrix{
 \cO_X \ar[r] &  F^e_* \cO_X(\lceil (p^e - 1)B \rceil) \ar[r]^-{\phi} & \cO_X.
}
\]
Because all the sheaves above are reflexive and $X$ is normal, we can tensor{\footnote{on the regular locus, and extend as in Remark \ref{reflex}}}
with  $\cO_X(\lceil (p^f - 1) D \rceil$ to   obtain a  splitting
\[
 \xymatrix{
\cO_X(\lceil (p^f - 1) D \rceil ) \ar[r] &  F^e_* \cO_X(\lceil (p^e - 1)B \rceil + p^e \lceil(p^f - 1) D\rceil) \ar[r] & \cO_X(\lceil (p^f - 1) D \rceil).
 }
\]
Applying $F^f_*$ to this splitting, and then composing with $\psi$ we obtain the following splitting,
\[
 \xymatrix{
 \cO_X \ar[r] & F^{e+f}_* \cO_X(\lceil (p^e - 1)B \rceil + p^e \lceil (p^f - 1)D \rceil) \ar[r] & \cO_X
}
\]
However, we also note that
\[
\lceil (p^e - 1)B \rceil + p^e \lceil (p^f - 1)D \rceil \geq \lceil  (p^e - 1)B  + p^e  (p^f - 1)D \rceil
\]
which implies that we also have a splitting
\[
 \xymatrix{
 \cO_X \ar[r] & F^{e+f}_* \cO_X(\lceil (p^e - 1)B  + p^e (p^f - 1)D \rceil) \ar[r] & \cO_X
}
\]
If we then multiply $(p^e - 1)B + p^e(p^f - 1)D$ by ${1 \over p^{(e + f)} - 1}$, the proof of (i) is complete for the choice of $\epsilon $ given in line \ref{eps}.

Now, to prove (ii), we use Theorem
 \ref{LemmaSplitAlongAndComplementDivisorImpliesFRegular}.
  Choose an effective integral divisor $C$ whose support contains the support of $D$ and such that the pair $(X \setminus C, D|_{X \setminus C})$ is globally  $F$-regular. Since there exists a splitting of
\[
  \cO_X \rightarrow F^f_*\cO_X( \lceil (p^f - 1)B + C \rceil),
\]
 it follows  that the
 pair $(X, B + \frac{1}{p^f -1}C)$ is globally sharply $F$-split. Applying part (i) of the Lemma to the pairs
 $(X, B + \frac{1}{p^f -1}C)$ and $(X, D)$, we conclude that
\[
 (X, \epsilon(B + \frac{1}{p^f -1}C) + (1 - \epsilon)D)
\]
  is globally sharply $F$-split.
Re-writing, we have
\[
 (X, \epsilon B  + (1 - \epsilon) D + \epsilon' C)
\]
 is globally sharply $F$-split  for $\epsilon$ and $\epsilon'$ arbitrarily close to zero.

We now apply Theorem  \ref{LemmaSplitAlongAndComplementDivisorImpliesFRegular}
to the pair $(X, \Delta) =  (X, \epsilon B +  (1 - \epsilon) D)$. Restricted to $X\setminus C$, this pair is globally $F$-regular, and we've just shown that for sufficiently small $\epsilon'$, the pair $(X, \Delta+ \epsilon' C) $ is globally sharply $F$-split. Using Corollary \ref{alternate}, we conclude that $(X, \Delta)$ is globally $F$-regular.

Finally, note that because  of the explicit choice of  $\epsilon$
in line (3),  it is clear its denominator can be assumed not divisible by $p$.
\end{proof}

We are now ready to begin the proof of Theorem \ref{MainExistenceTheorem}. We first recall some important facts about duality, and describe the main idea.

\subsection{Frobenius and duality}
Let $X \overset{g}\rightarrow Y$ be a finite morphism of finite type schemes over an $F$-finite local ring with a dualizing complex.  If $X$ and $Y$ are Cohen-Macaulay and have the same dimension, one consequence of Grothendieck duality  is that for any coherent sheaf $\FF$ on $X$, there is a natural isomorphism
$$
g_* \sHom_{\cO_X}(\FF, \omega_X) \cong \sHom_{\cO_Y}(g_*\FF, \omega_Y).
$$
See \cite[Chapter III,\S 6 and Chapter V, Prop 2.4]{HartshorneResidues} for details, and Section \ref{canonical} for the definition of the canonical module.

It follows that  if $X$ and $Y$ are assumed normal instead of Cohen-Macaulay, we can make the same conclusion, by restricting to the Cohen-Macaulay locus (an open set whose compliment has codimension greater than two)  and then extending the isomorphism over the whole of $X$. (Cf. Discussion \ref{reflexeasy}.)

Now, if $X$ is a normal $F$-finite scheme of essentially finite type over an $F$-finite Gorenstein local ring where $X$ has canonical module $\omega_X$ (see Discussion \ref{canonical}), applying this to the Frobenius morphism yields
$$
F^e_* \sHom_{\cO_X}(\FF, \omega_X) \cong \sHom_{\cO_X}(F^{e}_*\FF, \omega_X).
$$

\subsection{Divisors dual to splittings}\label{key}
The key idea of the proof of Theorem \ref{MainExistenceTheorem} is the following.
Fix an effective divisor $D$ on a normal $F$-finite scheme $X$, and consider the   natural map $\cO_X \rightarrow F_*^e\cO_X(D)$. If this map splits, we have a composition
\begin{equation}\label{a}
\cO_X \rightarrow F_*^e\cO_X(D) \overset{s}\rightarrow \cO_X,
\end{equation}
which is the identity map on $\cO_X$. Applying the duality functor $\sHom_{\cO_X}(-, \omega_X)$, we have
a dual splitting
$$
\omega_X \leftarrow F_*^e \sHom_{\cO_X}(\cO_X(D), \omega_X) =F_*^e(\cO_X(-D +K_X))  \overset{s^{\vee}}\leftarrow \omega_X.
$$
Now, twisting again by $\omega_X^{-1}$ and using the projection formula for $F^e$ on the
middle term,{\footnote{As usual, we work with reflexive sheaves by restricting to the regular locus, where all reflexive sheaves are invertible, and then extend uniquely to reflexive sheaves on $X$.}}   we arrive at a new splitting
\begin{equation}\label{b}
\cO_X \leftarrow F_*^e\cO_X((1-p^e)K_X-D) \overset{s^{\vee}}\leftarrow \cO_X.
\end{equation}
Note that the functor we actually applied to the splitting (\ref{a}) to get (\ref{b}) can be succinctly described as $\sHom_{\cO_X}(\blank, \cO_X)   $.

Now, the map $s^{\vee}$ determines  a section of $ F_*^e\cO_X((1-p^e)K_X-D)$,  namely the image of the element $1$, and this section in turn determines a uniquely defined  effective Weil  divisor $D'$ on $X$. This divisor $D'$ is linearly equivalent to $(1-p^e)K_X-D$.  In fact, under the induced
 identification of $\cO_X((1-p^e)K_X-D)$ with the subsheaf $\cO_X(D')$ of the function field $K$, the  map $s^{\vee}$ recovers the natural map
 $$\cO_X \overset{s^{\vee}}\rightarrow F^e_*\cO_X(D'),$$
which we have just shown to split.  We are now ready to combine this technique with Lemma \ref{CorollaryCombinationOfFRegularAndFPureImpliesFRegular} into a construction  of the divisor
 $\Delta$ needed for Theorem \ref{MainExistenceTheorem}.

\begin{proof}[Proof of Theorem \ref{MainExistenceTheorem}.]
First, without loss of generality, we may assume that the $\bQ$-divisor $B$  has no denominators divisible by $p$. Indeed,
 choose  $\tld B \geq B$  as in Proposition \ref{CorCanReplaceBByNice}.  If we can construct $\tld \Delta$ satisfying the conclusion of Theorem \ref{MainExistenceTheorem} for the pair $(X, \tld B)$, then $\Delta = \tld \Delta + (\tld B - B)$ satisfies the conclusion of Theorem \ref{MainExistenceTheorem} for the initial pair $(X, B)$.  Therefore, we may assume that no coefficient of $B$ has $p$ in its denominator, or put differently, that the coefficients of  $(p^e -1)B $ are all integers, for all sufficiently large and divisible $e$.

We first prove statement (ii), which follows quite easily from the dualizing technique described in Paragraph \ref{key} above.  Suppose that $(X, B)$ is globally sharply $F$-split.  Consider a splitting
\[
 \xymatrix{
\cO_X \ar[r] & F^e_* \cO_X \ar[r] & F^e_* \cO_X((p^e - 1)B) \ar[r]^-{\phi} & \cO_X
}
\]
where $(p^e - 1)B$ is an integral divisor.
Apply $\sHom_{\cO_X}(\blank, \cO_X)$ to this splitting, as discussed in Paragraph \ref{key}.  We then obtain the following splitting,
\[
 \xymatrix{
\cO_X & \ar[l] F^e_* \cO_X((1-p^e)K_X) & \ar[l] F^e_* \cO_X((1-p^e) (K_X + B)) & \ar[l]_-{\phi^{\vee}} \cO_X.
}
\]
The image of $1$ under $\phi^{\vee}$ determines a divisor $D'$ which is linearly equivalent to $(1-p^e) (K_X + B)$.  This produces a composition
\begin{equation}
\label{EqnReverseSplitting}
 \xymatrix{
\cO_X & \ar[l] F^e_* \cO_X(D'+(p^e - 1)B) & \ar[l] F^e_* \cO_X(D') & \ar[l]_-{\phi^{\vee}} \cO_X.
}
\end{equation}
Set $\Delta_1 = {1 \over p^e - 1} D'$.  Then the pair $(X, B + \Delta_1)$ is globally sharply $F$-split
with the splitting given by  Equation (\ref{EqnReverseSplitting}). But also, it is log Calabi Yau, since
$$K_X + B + \Delta_1\sim_{\bQ} K_X + B + {1 \over p^e -1}(1 - p^e)(K_X + B) = 0.$$
This completes the proof of (ii).

More work is required to prove (i).
Suppose that $(X, B)$ is globally $F$-regular.  Then it is also globally sharply $F$-split, and we may pick  $\Delta_1$ as in (ii).   Choose $H$ to be a very  ample effective divisor such that $\Supp \Delta_1\subseteq \Supp H$.  Consider a splitting
\[
 \xymatrix{
\cO_X \ar[r] & F^f_* \cO_X(H)  \ar[r] & F^f_* \cO_X((p^f - 1)B + H) \ar[r]^-{\psi} & \cO_X,
}
\]
such that $(p^f - 1)B$ is integral.
Apply $\sHom_{\cO_X}(\blank, \cO_X)$ to this splitting to obtain a dual splitting,
\begin{equation}\label{c}
 \xymatrix{
\cO_X & \ar[l] F^f_* \cO_X((1-p^f) K_X - H) & \ar[l] F^f_* \cO_X((1-p^f) (K_X + B) - H) & \ar[l]_-{\psi^{\vee}} \cO_X
}
\end{equation}
The image of $1$ under $\psi^{\vee}$ determines a divisor $D''$ which is linearly equivalent to $(1-p^f) (K_X + B) - H$.  Set $\Delta_2 = {1 \over p^f - 1} D''$.  Note that
\[
K_X + B + \Delta_2 \sim_{\bQ}  {-1 \over p^f - 1} H
\]
which is anti-ample.  Also note that the splitting in line (\ref{c}) demonstrates the
 pair $(X, B+  \Delta_2)$ to be  globally sharply $F$-split. Even better, line (\ref{c}) also demonstrates
$(X, B+  \Delta_2 + \frac{1}{p^f - 1}H)$ to be globally sharply $F$-split.

We now make use of  Lemma \ref{CorollaryCombinationOfFRegularAndFPureImpliesFRegular} to complete the proof.  In addition to the globally $F$-regular pair $(X, B)$, we have constructed divisors $\Delta_1$ and $\Delta_2$  satisfying
\begin{itemize}
\item[(i)]  $(X, B + \Delta_1) $ is globally sharply $F$-split with $K_X + B + \Delta_1 \sim_{\bQ} 0$;  and
\item[(ii)]  $(X, B + \Delta_2) $ is globally sharply $F$-split with $K_X + B + \Delta_2$  anti-ample.
\item[(iii)] $(X, B + \Delta_2 + \delta H ) $ is globally sharply $F$-split for some small positive $\delta$.
\end{itemize}

Now we apply Lemma \ref{CorollaryCombinationOfFRegularAndFPureImpliesFRegular}(i) to the divisors described in (i) and (iii) above.  We thus fix positive rational numbers $\epsilon_1, \epsilon_2 $, with $\epsilon_1 + \epsilon_2 = 1$  such that
$$
 (X, \epsilon_1 (B + \Delta_1)  + \epsilon_2(B + \Delta_2 + \delta H)) = (X, B+  \epsilon_2  \Delta_2 + \epsilon_1 \Delta_1 + \epsilon_2 \delta H)
$$
is globally sharply $F$-split.
Since the support of $\Delta_1$ is contained in the support of $H$, it follows from Lemma \ref{simple} that
\begin{equation}\label{2}
(X, B+  \epsilon_2  \Delta_2 + (\epsilon_1 + \epsilon') \Delta_1 )
\end{equation}
is  globally sharply $F$-split for some small positive $\epsilon'$. But also $(X, B+  \epsilon_2  \Delta_2) $ is globally $F$-regular, as one sees by applying Lemma
 \ref{CorollaryCombinationOfFRegularAndFPureImpliesFRegular}(iii)  to the globally $F$-regular pair $(X, B)$   and the globally sharply $F$-split  pair $(X, B + \Delta_2) $.

Finally, another application of Lemma  \ref{CorollaryCombinationOfFRegularAndFPureImpliesFRegular}(iii),
this time to the globally $F$-regular pair $(X, B+  \epsilon_2  \Delta_2) $ and the globally sharply $F$-split pair of line (\ref{2}), implies that
$
(X, B+  \epsilon_2  \Delta_2 + \epsilon_1  \Delta_1)
 $
is globally $F$-regular.  Set $\Delta = \epsilon_1 \Delta_1 +  \epsilon_2 \Delta_2 $. We conclude that the pair $(X, B + \Delta)$ is globally $F$-regular, and
$$
K_X + B + \Delta = \epsilon_1(K_X + B + \Delta _1) +  \epsilon_2(K_X + B + \Delta _2)
$$
is anti-ample (from (i) and (ii) just above). This completes the proof of (i) and hence Theorem \ref{MainExistenceTheorem}. \end{proof}

\begin{remark} The divisor $\Delta$ constructed is such that
 there exists some $e > 0$ such that $(p^e - 1) (B + \Delta)$ is an integral divisor and such that  $(p^e - 1)(K_X + B + \Delta)$ is an integral Cartier divisor.
In particular, the characteristic $p$ does not show up in the denominator of any coefficient of
 $\Delta + B$\end{remark}

\begin{remark}
The proof of Theorem \ref{MainExistenceTheorem} actually shows that if $(X, B)$ is globally $F$-regular (but $X$ is not necessarily quasi-projective), then for every effective divisor $D$ on $X$, there exists an $\epsilon > 0$ and a $\bQ$-divisor $\Delta$ on $X$ such that $(X, B + \Delta)$ is globally $F$-regular and $-K_X - B - \Delta \sim_{\bQ} \epsilon D$.
\end{remark}

\begin{remark}
Variants and special cases of the  trick to construct $\Delta$ have been used before. For example, it is discussed in  \cite[Section 2]{MehtaRamanathanFrobeniusSplittingAndCohomologyVanishing} as a way for developing criteria for $F$-splitting of smooth projective varieties.  Similar ideas are used later by Hara and Watanabe in \cite[Second proof  of Theorem 3.3]{HaraWatanabeFRegFPure}).    Also see \cite{SchwedeFAdjunction} for other interpretations of such $\Delta$ in a local setting.
\end{remark}

\section{Log Fano implies globally $F$-regular type in characteristic zero}
\label{SectionLogFanoImpliesFType}

We now consider the converse: is every log Fano variety globally $F$-regular? N\"aively posed, the answer is ``No". For example, the  smooth cubic surface in $\bP^3$ defined by the equation $w^3+ x^3 + y^3 + z^3 $ is easily checked to be log Fano in every characteristic (except of course $3$), but {\it not\/} globally $F$-regular, or  even $F$-split,  in characteristic $2$.

On the other hand, this cubic surface  is globally $F$-regular over any field whose characteristic is at least 5. In other words,  if we consider the cubic surface over any field of characteristic zero, it has {\it globally $F$-regular type.\/} The converse to Theorem 1.1 does hold in characteristic zero:

\begin{theorem}\label{char0converse}  Let $X$ be a normal projective variety over a field of characteristic zero.
If $(X, \Delta)$ is a Kawamata log terminal pair such that $K_X + \Delta$ is anti-ample, then $(X, \Delta)$ has globally $F$-regular type.
\end{theorem}

Roughly speaking, a pair $(X, \Delta)$ has globally F-regular type if, for almost all ``reductions mod $p$", the corresponding pairs over the various finite fields are globally F-regular. We do not include the definition here, but instead refer to the general discussion in  \cite[Section 1 and Definition 4.2]{SmithVanishingSingularitiesAndEffectiveBounds}, or to the related definitions of ``dense F-pure type" and so on in  \cite{HaraWatanabeFRegFPure}.

\begin{remark} \label{RemarkLogCYImpliesFSplit}
We also expect that if $(X, \Delta)$ is a log canonical pair such that $K_X + \Delta$ is $\bQ$-trivial, then $(X, \Delta)$ has  dense globally sharp $F$-split type, meaning that for a dense set of primes among all  reductions mod $p$, the corresponding pair is globally sharply F-split.  This would follow if we knew that every log canonical pair had dense $F$-pure type, see \cite[Problem 5.1.2]{HaraWatanabeFRegFPure}.
 \end{remark}

The proof of Theorem \ref{char0converse} is a non-trivial generalization of some of the main ideas of \cite{SmithGloballyFRegular}. The point is to reduce these statements about the global geometry of a projective variety $X$ to statements about the singularities at the vertex of the cone over $X$ for some projective embedding. We begin with a review of section rings for a projective variety.

\subsection{Section rings}\label{sectionrings}

Fix a normal (irreducible) projective variety $X$ over a field. For any ample invertible sheaf $\sL$ on $X$, we have a corresponding \emph{section ring\/}
\begin{equation}\label{sectring}
S = S(X, \sL) := \bigoplus_{i \in \bZ} H^0(X, \sL^i).
\end{equation}
The ring $S$ is an $\bN$-graded  normal domain, finitely generated over the field $k  = S_0$, such that  the projective $k$-scheme $\Proj S$ recovers $X$. See \cite[4.5.1]{EGA}. We denote its unique homogeneous maximal ideal, generated by the elements of positive degree, by $\frak m$. The ``punctured spectrum" of $S$ can be viewed as a $k^*$-bundle over $X$ via a natural flat``quotient map"
 \begin{equation}\label{puncture}
  \sU = \Spec S \setminus \frak m \overset{q}\rightarrow X
 \end{equation} where $k^* $ denotes the scheme $ \Spec k[t, t^{-1}]$.
 (This is more or less by definition when $S$ is generated in degree one; the general case is discussed, for example, in \cite[Prop 2.1(4)]{HyrySmithCoreVersusGradedCore}.)

There is a well-known correspondence between (isomorphism classes of) coherent sheaves on $X$ and finitely generated graded $S$-modules, ``up to isomorphism of the tails".
Under this correspondence, the invertible sheaves $\sL^n$ correspond to the
graded $S$-modules $S(n)$,  by which we  denote the rank one free  $S$-module whose generator has degree $-n$.  More generally,
 each finitely generated graded $S$-module $M$ determines a coherent sheaf $\widetilde M$ on $X$, whose value on the basic open set $D_+(f)$ is $[M \otimes_S S_f]_0$ (using the notation of \cite[Chapter II, Section 5]{Hartshorne}).
Of course, if $M$ and $M'$ agree in large degree, it follows that $\widetilde M = \widetilde M'$. On the other hand, given a coherent sheaf $\sF$  on $X$,  there is a {\it unique\/}
  {\it saturated\/} graded $S$-module
  $$
M_{\sF} = M(X, \sL, \sF) := \bigoplus_{i \in \bZ} H^0(X, \sF\otimes \sL^i)
$$
 in the class of all finitely generated graded $S$-modules determining $\sF$ on $X$.{\footnote{Proof:    Let $\mathcal U = \Spec S \setminus \frak m$, and let $M$ be any finitely generated graded $S$-module producing the coherent sheaf $\sF$ on $X$. Consider the exact sequence
 $$
0 \rightarrow H^0_{\frak m}(M) \rightarrow M \overset{restr.}\rightarrow  \Gamma(\mathcal U, \widetilde M) =
  \bigoplus_{i \in \N} H^0(X, \widetilde{M(i)})  \rightarrow H^1_{\frak m}(M) \rightarrow H^1(\Spec S, \widetilde M) = 0.
  $$
  Here, the equality  $\Gamma(\mathcal U, \widetilde M) =
  \bigoplus_{i \in \N} H^0(X, \widetilde{M(i)})$ follows by direct computation of the two sheaves on the open sets $D(f) \subset \Spec S$ and $D_+(f) \subset X$, for any homogeneous element $f \in S$ of positive degree (notation as in \cite{Hartshorne}).
  Since $\widetilde{M(i)} =
  \sF\otimes\sL^i$ on $X$, we see that $M$ is saturated (that is, that $H^1_{\frak m}(M) = 0$) if and only if it is isomorphic (as a graded module) to $M_{\sF}$.}}
  By saturated, here, we mean that depth of $M$ on the unique homogeneous maximal ideal $\frak m$ is at least two, or equivalently, that the local cohomology module $H^1_{\frak m}(M) = 0$.
Under this correspondence, maps between coherent sheaves on $X $ correspond to {\it degree preserving\/} maps of graded $S$-modules.
Importantly, the graded $S$-module corresponding to the canonical sheaf $\omega_X$ on $X$ is the graded canonical module $\omega_S$ for $S$, see \cite[p. 420]{SmithFujitasFreeness}. Note that reflexive modules are always saturated, so to check this, it suffices to verify that $\omega_S$ determines the sheaf $\omega_X$.

For more about the general theory of section rings, see \cite[II \S 5]{Hartshorne} (which covers the special case where $S$ is assumed to be generated by degree one elements), \cite{DemazureNormalGradedRings}, \cite[II \S 2]{EGA}, \cite{SmithVanishingSingularitiesAndEffectiveBounds}, \cite{SmithFujitasFreeness}, or \cite{HyrySmithCoreVersusGradedCore}.

\subsection{Divisors and section rings}\label{divoncone}
 Fix any section ring $S = S(X, \mathcal L)$ for the normal projective  variety $X$. For any
prime divisor $D$ on  $X$, there is a corresponding prime divisor $D_S$ on $\Spec S$.
There are three  useful ways to describe $D_S$:
\begin{itemize}
\item[(1)] Thinking of $D$ as corresponding to a height one homogeneous prime ideal of the graded ring $S$,  $D_S$ is the  divisor of $\Spec S$ corresponding to the same prime ideal.
\item[(2)] The divisor $D_S$ is the unique Weil divisor on $\Spec S$ such that the reflexive sheaf $\cO_S(D_S)$ determines the coherent sheaf $\cO_X(D)$. Put differently,
$$
\cO_S(D_S) = M_{\cO_X(D)} = \bigoplus_{n\in \bZ} H^0(X,  \cO_X(D)\otimes \mathcal L^n),
$$ as submodules of the fraction field of $S$.
\item[(3)]  The divisor $D_S$
is the unique  Weil divisor on $\Spec S$ agreeing on the punctured spectrum with the pullback $q^*D$, where $q: \Spec S \setminus \bm \rightarrow X$ is the $k^*$-bundle described in Paragraph \ref{sectionrings}. (Because $q$  is flat, Weil divisors can be pulled back to Weil divisors, whether or not they are Cartier.)
 \end{itemize}
It is easy to see that these are all equivalent descriptions of the same divisor; indeed, one need only observe that they all agree on the punctured spectrum of $\Spec S$, an open set of a normal irreducible scheme whose compliment has codimension two or more.   Clearly this procedure extends by linearity to any divisor on $X$, and even to any $\bQ$-divisor $\Delta$ on $X$, to produce a well-defined
 $\bQ-$divisor $\Delta_S$ on $\Spec S$. In particular,
for any Weil divisor $A$ on $X$, the reflexive sheaf $\cO_S(A_S)$ is the unique saturated graded $S$-module corresponding  to the sheaf $\cO_X(A)$.

We will eventually prove Theorem \ref{char0converse} by reducing to a corresponding statement about section rings. First, we state a characterization of globally $F$-regular pairs in terms of section rings, generalizing  \cite[Thm 3.10 (4)]{SmithGloballyFRegular}.

\begin{proposition}\label{Fregsectionring}
Let $X$ be  a normal projective variety $X$ over an $F$-finite field of prime characteristic.  For {\it any choice of ample invertible sheaf\/} $\sL$,  let $S$ denote the corresponding section ring.
Then
\begin{itemize}
\item[(1)] The pair $(X, \Delta)$ is globally $F$-regular if and only if $(\Spec S, \Delta_S)$ is globally $F$-regular.
\item[(2)] The pair
 $(X, \Delta)$ is globally sharply $F$-split if and only if $(\Spec S, \Delta_S)$ is globally sharply $F$-split.
\end{itemize}
\end{proposition}

Similarly, in characteristic zero, the  global geometry of $(X, \Delta)$ is  governed by the singularities of an affine cone over it. The following generalizes \cite[Proposition 6.2]{SmithGloballyFRegular} also see \cite[Proposition 4.38]{FujinoIntroductionToLMMPForLCBook}:

\begin{proposition}\label{coneoverlogFano}
Let $X$ be a normal projective variety over a field of characteristic zero. Then
\begin{itemize} \item[(1)]
The pair $(X, \Delta)$ is log Fano if and only if $X$ admits a section ring $S$ such that the
corresponding pair $(\Spec S, \Delta_S)$ is Kawamata log terminal.
\item[(2)] If the pair $(X, \Delta)$ is log Calabi-Yau then for every section ring $S$ of $X$,  the corresponding pair $(\Spec S, \Delta_S)$ is log canonical.  Conversely, if for some section ring $S$ the pair  $(\Spec S, \Delta_S)$ is log canonical, then there exists a $\bQ$-divisor $B$ on $X$ such that $(X, \Delta+B)$ is log Calabi-Yau.
\end{itemize}
\end{proposition}

\begin{remark}
We can not hope to show that {\it every\/} section ring of a log Fano variety  is Kawamata log terminal, except in the special case where the rank of the Picard group is one. Indeed, suppose $-K_X - \Delta = H$ is ample, and let $H'$ be any ample divisor such that $H$ and $H'$ represent $\bQ$-independent classes in Pic$(X) \otimes \bQ$. If $S$ is the section ring with respect to $H'$, then  the divisor $K_S + \Delta_S$ can {\it never\/} be $\bQ$-Cartier. Indeed, since $S$ and $K_S + \Delta$ are graded, if   $\cO_S(n(K_S + \Delta_S))$ is locally free, then it is free, and hence isomorphic to  $\widetilde{S(m)} =
\cO_S(mH')$ for some integer $m$. But this would imply $n(K_X + \Delta) = -n H \sim m H'$, a contradiction to the assumption that $H$ and $H'$ are independent.
\end{remark}

Before proving these propositions, we put them together into a proof of Theorem
 \ref{char0converse}.

\begin{proof}[Proof of Theorem \ref{char0converse}]
Let $(X, \Delta)$ be a log Fano pair.  Using Proposition \ref{coneoverlogFano}, construct a section ring $S = S(X, \sL)$ such that the corresponding  pair $(\Spec S, \Delta_S)$ is Kawamata log terminal.  In particular, $(\Spec S, \Delta_S)$ has globally $F$-regular type, by \cite[Corollary 3.4]{TakagiInterpretationOfMultiplierIdeals}.  It now follows from Proposition \ref{Fregsectionring} that $(X, \Delta)$ has  globally $F$-regular type.
\end{proof}

\subsection{Frobenius and graded modules}
To prove Proposition \ref{Fregsectionring}, we need to understand how  Frobenius interacts with the correspondence between coherent sheaves on $X$ and graded modules on a section ring.
 First let $M$ be any $\bZ$-graded module over a $\bN$-graded ring $S$.
 There is a natural way to grade $F^e_*M$ by $ \frac{1}{p^e}\bZ$. Specifically, if $m \in M$ is a degree $n$ element, then viewed as an element of  $F^e_*M$, we declare its degree to be $\frac{n}{p^e}$. With this grading, the natural action of $S$ on
$F^e_*M$ respects the degrees: $s \in S_i$ acts on $m \in [F^{e}_*M]_{\frac{n}{p^e}}$ to produce
$s^{p^e}m \in  [F^{e}_*M]_{\frac{n}{p^e} + i}$. Furthermore, the graded  $S$-module  $F^e_*M$ decomposes into a direct sum of graded $S$-modules $[F^e_*M]_{\frac{i}{p^e}\Mod\bZ}$ for $i $ ranging from $0$ to $p^{e}-1$.

 Now, if $S$ is a section ring for the normal projective variety $X$, then
  the summand made of integer-degree elements
  $[F^e_*M]_{0\Mod \bZ}$ plays a special role:

\begin{lemma}\label{claim} With notation as above, if  $M$ is the unique saturated graded $S$-module corresponding to a coherent sheaf $\sF$ on $X, $ then  $[F^e_*M]_{0\Mod \bZ}$ is the  unique graded S-module corresponding to $F^e_*\sF$.
\end{lemma}

\begin{proof}[Proof of Lemma]
We need only compute that
 $$
 M_{F^e_*\sF} = \bigoplus_{i\in \bZ}H^0(X,  F^e_*\sF \otimes \sL^{i}) = \bigoplus_{i\in \bZ}H^0(X,  F^e_*(\sF \otimes \sL^{ip^e}))  =  \bigoplus_{i\in \bZ}H^0(X, \sF \otimes \sL^{ip^e}),
 $$
 where the second equality follows from the   projection formula.
 This is precisely the integer degree summand of
 $F^e_*M$ with the grading described above.
 \end{proof}

 \begin{proof}[Proof of Proposition \ref{Fregsectionring}]
The pair $(X, \Delta)$ is globally $F$-regular if and only if, for all effective $D$, there is an $e$ such that the natural map
$$
\cO_X \rightarrow F^e_*\cO_X(\lceil (p^e -1) \Delta \rceil + D)
$$ splits.
Tensoring this splitting with the sheaf of $\cO_X$-algebras $\oplus \sL^{i} $ and taking global sections, we have a degree preserving splitting of
\begin{equation}\label{*}
S \rightarrow  M_{F^e_*\cO_X(\lceil (p^e -1) \Delta \rceil + D)} =
\left[ F^e_*M_{\cO_X(\lceil (p^e -1) \Delta \rceil + D)}\right]_{0 \Mod \bZ}
\end{equation}
and hence a splitting of
$$
S \rightarrow  F^e_*M_{\cO_X(\lceil (p^e -1) \Delta \rceil + D)}.
$$
Since the saturated graded $S$-module
$M_{\cO_X(\lceil ( p^e -1) \Delta \rceil + D)}$ is $\cO_S(\lceil (p^e -1) \Delta_S \rceil + D_S)$,  this splitting is also a splitting of
\begin{equation}\label{conesplitting}
S \rightarrow   F^e_*\cO_S(\lceil ( p^e -1) \Delta_S \rceil + D_S).
\end{equation}
But to check  global regularity for the affine scheme pair $(\Spec S, \Delta_S)$, it suffices to check the requisite splittings for all {\it homogeneous\/} effective divisors $D_S$.  This can be seen to follow from, for example, Theorem  \ref{LemmaSplitAlongAndComplementDivisorImpliesFRegular}.  We conclude that
 if $(X, \Delta)$ is globally $F$-regular, so is $(\Spec S, \Delta_S)$.

The converse statement follows from reversing the argument:  given a splitting of the map
 from line (\ref{conesplitting}) above, it can be assumed to be homogeneous. (Indeed, for any finitely generated $S$-module $N$, the module $\Hom_S(N, S)$ is graded, so any splitting $s$ would be a sum of its homogeneous components. Since the splitting takes $1$ to $1$, the degree zero component of the splitting is also a splitting.)
Since the splitting factors through the integer degree part, we also have a (degree-preserving) splitting of the map in line (\ref{*}). Because this induces a splitting of the corresponding map of sheaves on $X$, the proof of the converse direction is also complete.

Finally, we get the corresponding statement about global sharp $F$-splitting with the exact same argument, taking $D$ (and hence $D_S$) to be the zero divisor.
\end{proof}

We turn to the proof of  Proposition \ref{coneoverlogFano}.
The following lemma simplifies our argument by reducing us to the case where the section ring is generated in degree one.
\begin{lemma}\label{etale}
 Let $\mathcal L$ be an invertible sheaf on a normal projective variety $X$, and let $S$ and $S^{(k)}$ denote the section rings with respect to $\mathcal L$ and $\mathcal L^k$, respectively.
  Then:
 \begin{enumerate}
 \item
 For any positive integer $k$, the natural morphism (induced by the  inclusion $S^{(k)} \hookrightarrow S$)
$$
\Spec S \rightarrow \Spec  S^{(k)}
$$
is \'etale on the punctured spectra (that, is, on the open sets complimentary to the unique homogeneous maximal ideals of $S$ and $S^{(k)}$).
\item If   $\Delta_{S^{(n)}}$ and $\Delta_S$ are $\mathbb Q$-Weil divisors on
 $\Spec S^{(n)}$ and $\Spec S$ corresponding to the same  $\mathbb Q$-divisor on $X$ (Cf. Discussion \ref{divoncone}), then the pair $(\Spec S, \Delta_S)$ is klt if and only if
the pair $(\Spec S^{(k)}, \Delta_{S^{(k)}})$ is klt. An analogous statement holds  for log canonical pairs.
\end{enumerate}

\end{lemma}
\begin{proof}
(1)
Choose an affine open cover of $X$ which trivializes $\mathcal L$. Thinking of $X = \Proj S^{(k)}$, we can assume that the sets in the cover have the form $D_+(x)$ for some $x \in S^{(k)}$. Here $D_+(x)
= \Spec  A$, where $A$ is the ring $\left[S^{(k)}[\frac{1}{x}]\right]_0 = \left[S[\frac{1}{x}]\right]_0$,  the notation from \cite[II \S5]{Hartshorne}.
We will show that the inclusion
$$
S^{(k)}[\frac{1}{x}] \subset S[\frac{1}{x}]
$$
is \'etale.  Since the $D(x_i)$ cover the punctured spectra of $S$ and $S^{(k)}$, this will show that
$$
\Spec S \rightarrow \Spec S^{(k)}
$$ is \'etale on the punctured spectra.

Since $\mathcal L$ is trivial on $D_+(x)$, we can fix some generator $t$, which is a degree one element of $ S[\frac{1}{x}]$. Note that  all the powers of $\mathcal L$ are trivial on $D_+(x)$ and that $t^n$ generates $\mathcal L^n$ on $D_+(x)$ for all integers $n$.  Thinking of $X$ as $\Proj S$,  the sheaves $\mathcal L^n$ correspond to the graded $S$-modules $S(n)$, so we have
$$
S[\frac{1}{x}] = \bigoplus_{n \in \bZ} \mathcal O_X(D_+(x)) t^n = A[t, t^{-1}].
$$
On the other hand, thinking of $X$ as $\Proj S^{(k)}$, we have
$$
S^{(k)}[\frac{1}{x}] = \bigoplus_{n \in \bZ} \mathcal O_X(D_+(x)) t^{kn} = A[t^k, t^{-k}].
$$
So the inclusion
$
S^{(k)}[\frac{1}{x}] \subset S[\frac{1}{x}]
$ can be identified with the inclusion
$$
A[t^k, t^{-k}] \subset A[t, t^{-1}],
$$ which is obviously \'etale, as it is given by the polynomial $T^k - a$, where $a = t^k \in  A[t^k, t^{-k}]$,  whose derivative is invertible.

(2).
Because  $S$ has dimension at least two, the finite map $$
f: \Spec S \rightarrow \Spec S^{(k)}
$$ is \'etale in codimension one. We  now easily check that both $f^*K_{S^{(k)}} = K_{S}$ and  $f^*(\Delta_{S^{(k)}}) = \Delta_S$.
We can now compare discrepancies directly, or simply deduce the desired claim  from
\cite[5.20 (4)]{KollarMori}.
\end{proof}

\subsection{Blowups of section rings and the tautological bundle}\label{naturalconstruction}
We now set up some notation and recall a basic construction we use in the proof of Proposition \ref{coneoverlogFano}; for details see \cite[8.7]{EGA} or  \cite[Section 6.6]{HyrySmithOnANonVanishingConjecture}.
 Let $S$ be a section ring of a normal projective variety $X$. For simplicity, we assume that
$S$ is generated in degree one (which amounts to assuming $\sL$ is very ample and for a projectively normal embedding).{\footnote{Though we don't need it here, everything in this section, including Diagram (\ref{EqnTautologicalBundleDiagram}),  holds true for any section ring if instead of blowing up $\frak m$, we take $Y$ to the be the graded blowup $\Proj S^{\natural}$ of $\Spec S$; See \cite[Section 6.6]{HyrySmithOnANonVanishingConjecture}.}}

Let $Y \overset{\pi}\rightarrow \Spec S$ denote the blowup of the cone $\Spec S$ at the vertex $\frak m$. The exceptional divisor of $\pi$  is naturally identified with the projective variety $X$;
 under this identification, restricting to the exceptional set, the sheaf $\cO_Y(-X)$ becomes  the ample invertible sheaf $\sL$.

The scheme $Y$ can also  be interpreted as  the total space of the tautological bundle for the embedding of $X$ given by $\sL$. That is, there is a natural isomorphism $Y \cong \Spec_{\cO_X}{(\oplus_{i \geq 0}\sL^i)}$, and the natural projection $Y \overset{\eta}\rightarrow  X$ defines a line bundle on $X$ whose sheaf of sections recovers $\sL^{-1}$.   Thinking of $Y$ in this way,  the inclusion of the exceptional divisor $X$ in
$Y$ becomes the inclusion of the zero section of the bundle in $Y$. In particular, the composition
$\pi^{-1}(\frak m) = X \overset{j}\hookrightarrow Y \overset{\eta}\rightarrow X $ is the identity map on $X$.
This is summarized by the following diagram
\begin{equation}
\label{EqnTautologicalBundleDiagram}
\xymatrix{
X \ar@{^{(}->}[r]^-{j} \ar[d] & Y \ar[r]^{\eta} \ar[d]^{\pi} & X \ar@{<-}[d]^{q} \\
\Spec (S / \bm) \ar@{^{(}->}[r] & \Spec S \ar@{<-^{)}}[r] & \Spec S \setminus \{\frak m\}
}
\end{equation}

\begin{proof}[Proof of Proposition \ref{coneoverlogFano}]
We begin with some general observations.
Fix any normal projective $X$ and effective $\mathbb Q$ divisor $\Delta$ on $X$. Let $S$ be any section ring, say with respect to the ample divisor $H$, and let $\Delta_S$ be the corresponding $\mathbb Q$-divisor on $\Spec S$. We want to prove that conditions on the  singularities of $(\Spec S, \Delta_S)$ are equivalent to conditions on the singularities {\it and positivity\/} of $(X, \Delta)$.
According to Lemma \ref{etale}, there is no loss of generality in assuming $S$ is generated in degree one.

First note that  $(X, \Delta)$ is klt (or lc) if and only if
$(\Spec S, \Delta_S)$ is klt (or lc) on the punctured spectrum. Indeed,
 $\Spec S \setminus \frak m =  \mathcal U$ is a $k^*$-bundle over $X$, so
  this follows  directly from the definitions after trivializing $q: \mathcal U \rightarrow X$.{\footnote{Alternatively, one can make a general statement about the behavior of pairs under smooth morphisms.}} The issue is then that conditions on the singularities of $(\Spec S, \Delta_S)$ at the point $\frak m$ correspond to positivity conditions on the pair $(X, \Delta).$

Suppose that $K_S + \Delta_S$ is $\mathbb Q$-Cartier on $\Spec S$.  This means there is a (positive) integer $n$ such that the reflexive $S$-module corresponding to the integral Weil divisor $n(K_S + \Delta_S)$ is locally free of rank one. But, since this module is graded, locally free implies free, and so the unique reflexive graded $S$-module corresponding to $n(K_S + \Delta_S)$ is isomorphic to $S(t)$, for some degree shift $t \in \mathbb Z$.
In particular, $K_X + \Delta \sim_{\bQ} \frac{t}{n} H $ is either ample, anti-ample or $\mathbb Q$-trivial, depending on whether $t$ is positive, negative or zero.

Now, let $Y \overset{\pi}\rightarrow \Spec S$ be the blow up of $\Spec S$ at the maximal ideal $\frak m$, as discussed in Paragraph \ref{naturalconstruction}, with notation as in
Diagram (\ref{EqnTautologicalBundleDiagram}).
 We claim that
 \begin{equation} \label{calc}
\pi^*(K_S + \Delta_S) \sim_{\bQ} K_Y + \Delta_Y + (1+ \frac{t}{n}) X,
\end{equation}
where  $\Delta_Y$ is  the birational transform $\pi^{-1}_*\Delta_S$ of $\Delta_S$ on $Y$.
To prove this, fix $K_Y$ such that $\pi_*K_Y = K_S$ and write
$$
\frac{1}{n}\pi^*(n(K_S + \Delta_S)) = K_Y + \Delta_Y + a X
$$
for some $a \in \bQ$.  Now restrict  to the exceptional divisor $X$.  Since ${X}|_{X} = -H$, the adjunction formula gives
\begin{equation}\label{discrep}
 0 \sim_{\bQ}  (K_Y + X)|_{X} +  {\Delta_Y}|_{X} + (a-1) X|_{X} \sim_{\bQ} K_X + \Delta + (a-1)(-H),
  \end{equation}
  after we've verified that  $ {\Delta_Y}|_{X} = \Delta$. To this end, note that
   since $\Delta_S$ is the unique divisor agreeing with $q^*\Delta$ on
$\Spec S \setminus \{\frak m\} \cong Y \setminus X$, its birational transform on $Y$ is the closure of $q^*\Delta$ in $Y$, which is to say, $\Delta_Y = \pi^{-1}_* \Delta_S = \eta^*\Delta$.
  This implies that
$$
 {\Delta_Y}|_{X} = j^*\eta^*\Delta = (\eta \circ j)^* \Delta = \Delta,
$$
since $\eta \circ j$ is the identity map on $X$.
(Cf. Diagram (\ref{EqnTautologicalBundleDiagram}).)
  Remembering that $-K_X - \Delta \sim_{\bQ} \frac{t}{n} H$, formula (\ref{discrep}) yields $a = 1 + \frac{t}{n}$.

Now we are ready to prove Proposition \ref{coneoverlogFano}. Suppose that $(S, \Delta_S) $ is klt. In particular, from (\ref{calc}) it follows that
$(1+ \frac{t}{n}) < 1.$ This implies that $t $ is negative, and therefore that $K_X + \Delta \sim_{\bQ} \frac{t}{n}H$ is anti-ample. Similarly, if    $(S, \Delta_S) $ is lc,  then $t \leq 0$.  If $ t = 0$,  then
$K_X + \Delta$ is $\mathbb Q$-trivial, and we can take $B = 0$. Otherwise, $-K_X -\Delta$ is ample, and we  take $B$ to be $\frac{1}{m}$ for a general $G$ in the very ample linear system $|m(-K_X - \Delta_X)|$ for large $m$.  The pair $(X, \Delta + B)$  is log canonical by \cite[Lemma 5.17(2)]{KollarMori}.
This concludes the ``if direction" of the proof of Proposition \ref{coneoverlogFano} for both parts (1) and (2).

We now prove the converse. First, assume $(X, \Delta)$ is klt with $-K_X - \Delta$ ample.
Fix any positive integer $n$  such that $n(-K_X - \Delta)$ is an ample integral Cartier divisor $H$, and let $S$ be the section ring with respect to $H$. We need to show that $(\Spec S, \Delta_S)$ is klt in a neighborhood of $\frak m$. Without loss of generality, we assume $S$ is generated in degree one. Note that $K_S +  \Delta_S $ is $\bQ$-Cartier, with $n(K_S +  \Delta_S)$ corresponding to the free $S$-module $S(-1)$.

From (\ref{calc}), we see that
 $(\Spec S, \Delta_S)$  is klt at $\frak m$  if and only if  $(Y, \Delta_Y + (1 - \frac{1}{n}) X)$ is klt  in a neighborhood of $X$ \cite[Lemma 2.30]{KollarMori}.  The latter pair is  klt if the pair $(Y,   \Delta_Y + X)$ is purely log terminal; see \cite[Def 2.34]{KollarMori}.  By inversion of adjunction, the pair $(Y,   \Delta_Y + X)$ is purely log terminal  (in a neighborhood of $X$)
 since the pair $(X,  {\Delta_Y}|_{X}) = (X, \Delta)$ is klt
(see \cite[Theorem 5.50]{KollarMori}).
  Therefore $(\Spec S, \Delta_S)$ is Kawamata log terminal as desired.

Now, assume that $(X, \Delta)$ is log Calabi-Yau. Let $S$ be any section ring, without loss of generality generated in degree one.    We need to show $(\Spec S, \Delta_S)$ is log canonical in a neighborhood of $\frak m$.
  As above, it suffices to show that $(Y, \Delta_Y + X)$ is log canonical in a neighborhood of $X$, and this follows  from  inversion of adjunction on log canonicity since  the pair $(X, {\Delta_Y}|_{X}) = (X, \Delta)$
is log canonical \cite{KawakitaInversion}.
\end{proof}

\section{Properties of globally $F$-regular varieties}

In this section, we gather some results about globally $F$-regular varieties.
We begin with some general results, then treat projective $F$-regular varieties, proving a Kawamata-Viehweg vanishing theorem for them. We also gather a collection of (mostly well-known) examples of globally $F$-regular (log Fano) varieties.  Finally we consider the affine situation, and point out some applications to  tight closure theory.

\subsection{General properties}

First, we point out an openness result on $F$-regular pairs.

\begin{corollary}\label{CorGlobalFRegularityIsOpen}
Let $(X, \Delta)$ be a globally $F$-regular pair. Then for any effective divisor $D$ on $X$,
the pair $(X, \Delta+ \delta D)$ is globally $F$-regular for sufficiently small $\delta > 0$.
\end{corollary}

\begin{proof}
Because $(X, \Delta)$ is  globally $F$-regular, the splitting of
$$
\cO_X \rightarrow F^e\cO_X(\lceil p^e - 1 \Delta \rceil + D)
$$
shows that $(X, \Delta + \frac{1}{p^e -1}D)$ is globally sharply $F$-split.
Now Lemma \ref{CorollaryCombinationOfFRegularAndFPureImpliesFRegular} (iii) gives the desired conclusion.
\end{proof}

\begin{remark}[Cautionary Remark]  If $(X, \Delta_t)$ is a collection of globally $F$-regular pairs approaching (because the coefficients  are converging) some pair $(X, \Delta)$, it is not necessarily the case that the limit pair $(X, \Delta)$ is globally sharply $F$-split.  For example, if  $X = \Spec R$ and  $D$ is a Cartier divisor  whose  $F$-pure threshold $c$ of $(X, D)$ is a rational number with $p$ in the denominator, then $(X, cD)$ is not sharply $F$-pure.  See \cite{SchwedeSharpTestElements} for additional discussion.
\end{remark}

The next result shows that the property of global $F$-regularity is preserved under many common types of morphisms. For a partial converse, see \cite[Proposition 1.4]{HaraWatanabeYoshidaReesAlgebrasOfFRegularType}.

\begin{proposition}
\label{PropositionMehtaRamanathanVersion}
Suppose that $f : X \rightarrow Y$ is a morphism of normal varieties over an F-finite field of prime characteristic, and let $W$ be any open set of $X$ which maps into the smooth locus $U$ of $Y$ under $f$. Let $\Delta_Y$ be an effective $\bQ$-divisor on $Y$ such that no coefficient of $\Delta_Y$ has $p$ in the denominator, and let  $\Delta_W$ be its  pullback  to $W$.
  Then if the natural map $\cO_U \rightarrow f_* \cO_W$ splits and  if $(W, \Delta_W)$ is globally $F$-regular, then so is $(Y, \Delta_Y)$.
\end{proposition}

The main interest in Proposition \ref{PropositionMehtaRamanathanVersion} is when
$\Delta_Y$ (and $\Delta_W$) are zero.  Then the statement becomes
\begin{corollary}\label{corMR} Suppose that $X \overset{f}\rightarrow Y$ is a morphism of normal varieties over an F-finite field of prime characteristic. If
the natural map $\cO_Y \rightarrow f_* \cO_X$ splits, and  $X$ is globally F-regular,  then  $Y$ is also globally F-regular.
\end{corollary}

There are many common situations where the hypotheses of Corollary
\ref{corMR} hold. Indeed, the following two situations have been considered before:
\begin{itemize}
\item[(1)]
If $R \rightarrow S$ is a split inclusion of  domains, and $\Spec S$ is  globally $F$-regular, then so is $\Spec R$. This is essentially \cite[Theorem 3.1(e)]{HochsterHunekeTightClosureAndStrongFRegularity}.
\item[(2)]
If $X \overset{f}\rightarrow Y$ is a  proper morphism of normal varieties with $\cO_Y = f_*\cO_X$, and $X$ is globally $F$-regular, then so is $Y$.  This generalization to global $F$-regularity (from global $F$-splitting) of
\cite[Proposition 4]{MehtaRamanathanFrobeniusSplittingAndCohomologyVanishing} was observed in
\cite[Proposition 1.2(2)]{HaraWatanabeYoshidaReesAlgebrasOfFRegularType}.
\end{itemize}

In fact, the proof of Proposition \ref{PropositionMehtaRamanathanVersion} is essentially a jazzed up combination of the  arguments used in situations (1) and (2) above:

 \begin{proof}[Proof of Proposition \ref{PropositionMehtaRamanathanVersion}]
First note that if $(U, \Delta_U)$ is globally $F$-regular (respectively globally sharply $F$-split) then so is $(Y, \Delta_Y)$ because $Y \setminus U$ has codimension 2.  We define $\Delta_U$ to be $\Delta_Y|_U$.
Set $\phi : f_* \cO_W \rightarrow \cO_U$ to be a splitting of $\cO_U \rightarrow f_* \cO_W$.  Choose any effective divisor $C$ on $U$ (note that $C$ is Cartier since $U$ is smooth).  By assumption, there exists an $e > 0$ so that $(p^e - 1) \Delta_U$ is integral (and also Cartier).   Since $(W, \Delta_W)$ is globally $F$-regular (making $e$ larger if needed), we can choose a map
$ F^e_* \cO_W(f^*((p^e - 1)\Delta_U) + f^* C ) \overset{\psi} \rightarrow \cO_W$
such that the composition
$$
\cO_W \rightarrow F^e_* \cO_W(f^*((p^e - 1)\Delta_U) + f^* C)
\overset{\psi}\rightarrow  \cO_W
$$
is the identity map of $\cO_W$.  We have the following composition by the projection formula:
\[
\xymatrix@C=20pt{
\cO_U \ar[r] & F^e_* \cO_U((p^e - 1)\Delta_U + C) \ar[r] & f_* F^e_* \cO_W(f^*((p^e - 1)\Delta_U + C)) \ar[r]^-{f_* \psi} & f_* \cO_W \ar[r]^-{\phi} & \cO_U.
}
\]
It is easily seen to be the identity since the composition sends the global section $1$ to $1$.
To do the globally sharply $F$-split case, set $C = 0$.
\end{proof}

\begin{remark}
One could ask if a characteristic zero analog of (2) above holds.  That is, suppose that $f$ has connected fibers (or even that $\cO_Y \rightarrow f_* \cO_X$ splits) and there exists a $\Delta_X$ on $X$ such that $(X, \Delta_X)$ is log Fano.  Does there exist a divisor $\Delta_Y$ on $Y$ such that $(Y, \Delta_Y)$ is also log Fano?

We will show this if $f$ has connected fibers.  Our first step is to show that if $(X, \Delta_X)$ is log Fano, then for every Cartier divisor $D$ on $X$, there exists $\epsilon > 0$ and an effective $\bQ$-divisor $\Delta_D$ on $X$ such that $(X, \Delta_D)$ is Kawamata log terminal and $-(K_X - \Delta_D) \sim_{\bQ} \epsilon D$.  To construct $\Delta_D$, choose $\epsilon > 0$ such that $-K_X - \Delta_X - \epsilon D \sim_{\bQ} \delta A$ where $A$ is a general section of a very ample divisor and $\delta < 1$.   Set $\Delta_D = \Delta_X + \delta A$ and note that $(X, \Delta_D)$ is Kawamata log terminal (since $A$ was general).  Now choose $C$ to be an anti-ample divisor on $Y$ and let $D = f^* C$.  Choose $\epsilon > 0$ and $\Delta_D$ as above and apply \cite[Theorem 0.2]{AmbroTheModuliBDivisorOfAnLCFibration}.  Also compare with \cite{WisniewskiOnContractionsOfExtremalRaysOfFanos}.
\end{remark}

\begin{example} [Blowups of the projective plane]
Let $X_n$ be the blow-up of the projective plane at $n$  general points (over $\C$). Then $X$ has globally $F$-regular type if and only if $n \leq 8$.  Indeed, if $n \leq 8 $, then $-K_{X_n}$ is ample, so $X_n$ has globally $F$-regular type.  On the other hand, if $n = 9$, then $-K_{X_9} $ is not big, as can be verified by taking taking the anticanonical divisor on $\mathbb P^2$ to be the unique cubic through the nine points.  So $X_9$ is not globally $F$-regular, and nor is any $X_{n}$ for $n \geq 10$ by Corollary \ref{corMR}(1). \end{example}

\begin{example}[Ruled Surfaces] \label{ruled}  Let $X  \overset{\pi}\rightarrow C $ be a ruled surface over a smooth projective curve of genus $g$ (over a field of prime characteristic).  If  $g \geq 1$, then $X$ can not be globally $F$-regular.  Indeed, since $\cO_C \rightarrow \pi_*\cO_X$ is an isomorphism (that is, $X \rightarrow C$ has connected fibers), globally $F$-regularity of $X$ would force global $F$-regularity for $C$, by Corollary \ref{corMR}.  But of course, the only globally F-regular curve is $\bP^1$.  On the other hand, every ruled surface over $\bP^1$ is globally F-regular, because these are all  toric varieties;  see \cite{SmithGloballyFRegular}. In conclusion, a ruled surface is globally F-regular if and only if the base is rational.

The ample and effective cones for ruled surfaces can be worked out explicitly; see \cite[p. 70]{LazarsfeldPositivity1}.  Doing so, one can show that there are many ruled surfaces for which $-K_X$ is big, but for which there is no $\Delta$ with $(X, \Delta)$  ``controlled"  such that $-K_X - \Delta$ is ample. One interesting case is when $X = \bP(\cO_C \oplus \cO_C(-Q))$ is ruled over an elliptic curve, for some effective divisor $Q$ on $C$. In this case, we can find $\Delta$ such that
 $(X, \Delta)$ is log canonical and $-K_X - \Delta$ is big and nef.  Indeed, one can take $\Delta$ to be the class of a section of $X \rightarrow C$. On the other hand, using the explicit description of the ample and effective cones for such ruled surfaces, one can directly see there is no $\Delta$ such that $-K_X - \Delta$ is ample and $(X, \Delta)$ is log canonical.
\end{example}

\subsection{Kawamata-Viehweg vanishing for globally $F$-regular pairs}
We prove a version of the Kawamata-Viehweg vanishing for globally $F$-regular pairs.  The proof is similar to \cite[Corollary 4.4]{SmithGloballyFRegular}.

\begin{theorem}
\label{TheoremKawamataViehwegForGloballyFRegular}
Let $X$ be a normal projective variety over a field of prime characteristic.  Let $L$ be a Cartier divisor on $X$ such that $L \sim_{\bQ} M + \Delta$, where $M$ is a nef and big $\bQ$-divisor and the pair $(X, \Delta)$ is  globally $F$-regular.
Then $H^i(X, \cO_X(-L)) = 0$ for $i < \dim X$.
\end{theorem}
\begin{proof}
Because $L$ is big, we can fix $f \gg 0$ so that   there exists an  effective $E$ linearly equivalent to $p^f L$. By taking $f$ larger if necessary, we can also assume that for all large and sufficiently divisible $e$,
\begin{itemize}
\item[(1)] $p^f(p^e - 1) \Delta$ and $p^f(p^e - 1)M$ are integral,
\item[(2)]  $\cO_X(p^f(p^e - 1)L) \cong \cO_X(p^f(p^e - 1)(M + \Delta))$.
\end{itemize}

Since $M$ is nef and big, there exists an effective divisor $D$ such that $nM - D$ is ample for all $n \gg 0$; see \cite[Cor 2.2.7]{LazarsfeldPositivity1}.  Because $(X, \Delta)$ is globally $F$-regular, for all sufficiently large integers $g$, the map
\[
\cO_X \rightarrow F^g_* \cO_X( \lceil (p^g - 1)\Delta \rceil + D + E)
\]
splits.  By choosing $g$ large enough, we may assume that $g = f + e$ where $f$ is the fixed integer above and $e > 0$ is such that both (1) and (2) are satisfied above. Also, we can assume that  $p^f(p^e - 1)M - D$ is ample.
Therefore, the map
\[
\cO_X \rightarrow F^{e+f}_* \cO_X( p^f(p^e - 1)\Delta + D + E)
\]
splits since $p^f(p^e -1)\Delta \leq \lceil (p^{e+f} - 1)\Delta \rceil$. Tensoring (on the smooth locus, and extending as usual) with $\cO_X(-L)$ and taking cohomology, we have a splitting of
  the map
\[
H^i(X, \cO_X(-L)) \rightarrow H^i(X, F^{e+f}_* \cO_X( -p^{e+f} L + p^f(p^e - 1)\Delta + D + E)).
\] In particular, this map on cohomology is injective for all sufficiently large and divisible $e$.

However,
\[
\begin{split}
-p^{e+f} L + p^f(p^e - 1)\Delta + D + E =\\
 -(p^{e+f} - p^f)L - p^f L + p^f(p^e - 1)\Delta + D + E \sim \\
(-p^f(p^e - 1)M - p^f(p^e - 1)\Delta) + p^f(p^e -1)\Delta + D + \left( E-p^f L \right) \sim\\
 -p^f(p^e - 1)M + D
\end{split}
\]
which is anti-ample.  Therefore, $H^i(X, \cO_X( -p^{e+f} L + p^f(p^e - 1)\Delta + D + E))$ vanishes for $i < \dim X$ since $X$ is globally $F$-regular, by  \cite[Corollary 4.4]{SmithGloballyFRegular}, see also \cite{BrionKumarFrobeniusSplitting}.  Because of the injection above, it follows that $H^i(X, \cO_X(-L))$ vanishes, and the proof is complete.
\end{proof}

\subsection{Local consequences}

Theorem \ref{MainExistenceTheorem} has an interesting local interpretation.

\begin{corollary}
\label{CorLocalExistenceTheorem}
Suppose $R$ is normal, $F$-finite with $X = \Spec R$.  Then $X$ is $F$-pure (respectively strongly $F$-regular) if and only if there exists an effective $\bQ$-divisor $\Delta$ such that the pair $(R, \Delta)$ is $F$-pure (respectively strongly $F$-regular) and $K_R + \Delta$ is $\bQ$-Cartier with index not divisible by $p$.
\end{corollary}
\begin{proof}
It follows from the global statement if one observes that on an affine variety, every Cartier divisor is ample.
\end{proof}




\section{Further questions and observations}

As we observed, it is not difficult to find examples of log Fano varieties that are not globally $F$-regular, for certain low characteristic.  However, one can ask the following question:

\begin{question}
Suppose that $X$ has globally $F$-regular type, then does there exist an effective divisor $\Delta$ such that $(X, \Delta)$ is log Fano? Similarly, if  $X$ has globally sharply $F$-split type, then is $X$ log Calabi-Yau?
\end{question}

Although naively one might expect this to follow from Theorem \ref{MainExistenceTheorem},
the problem is that the $\Delta$ we construct there is not canonical and will probably depend on the characteristic.   On the other hand, many classes of varieties are known to be $F$-split (respectively, globally $F$-regular).  For example, for projective toric varieties, it is easy to directly construct a (toric) $\Delta$ independent of characteristic such that $K_X + \Delta$ is anti-ample and $(X, \Delta)$ is Kawamata log terminal. Similarly, Schubert varieties are known to be globally $F$-regular, see \cite{LauritzenRabenThomsenGlobalFRegularityOfSchuertVarieties}.  It is  thus  natural  to try to find divisors $\Delta$ such that $(X, \Delta)$ is log Fano for Schubert varieties, independent of the characteristic (this has recently been done by David Anderson and Alan Stapledon).  There are many other examples of globally F-regular varieties arising naturally in representation theory (see, for example,
 \cite{HeThomsenGeometryOfBTimesBOrbitClosures} or \cite{BrionThomsenFRegularityOfLarge})
and it would interesting to find natural geometric or representation theoretic descriptions of a divisor $\Delta$ showing them to be log Fano independent of the characteristic.



 Turning to local considerations, given \cite[Section 5]{DeFernexHacon}, it is natural to ask whether  a local (but non-affine) analog of Theorem \ref{MainExistenceTheorem} holds. That is,
\begin{question}
If $X$ is locally $F$-regular, does  there exist an effective $\bQ$-divisor $\Delta$ such that
$K_X + \Delta$
 is $\bQ$-Cartier and $(X, \Delta)$ is locally $F$-regular?  \end{question}

A related question in commutative tight closure theory is the following.
\begin{question}
Suppose that $R$ is a normal reduced $F$-finite ring.  Does there exist an effective $\bQ$-divisor $\Delta$ on $\Spec R$ such that $K_R + \Delta$ is $\bQ$-Cartier and $\tld \tau(R) = \tld \tau(R, \Delta)$?
\end{question}

For the definition of $\tilde \tau$ we refer to the paper \cite{LyubeznikSmithCommutationOfTestIdealWithLocalization} for the original boundary-less definition and to \cite{TakagiInterpretationOfMultiplierIdeals} for the general case, also see \cite{HaraTakagiOnAGeneralizationOfTestIdeals}.

The following question is natural in view of Lemma
\ref{CorollaryCombinationOfFRegularAndFPureImpliesFRegular} (and since the corresponding statement in characteristic zero is true).

\begin{question}
Suppose that the pairs $(X, \Delta_1)$ and $(X, \Delta_2)$ are sharply $F$-pure, then is it true that for a dense set of $\epsilon \in (0, 1)$, the pair $(X, \epsilon \Delta_1 + (1 - \epsilon) \Delta_2)$ is sharply $F$-pure?
\end{question}

It follows from Lemma \ref{CorollaryCombinationOfFRegularAndFPureImpliesFRegular} that there are infinitely many such $\epsilon$ and infinitely many of them accumulate around the points ${0}$ and $1$.  However, not all $\epsilon$ satisfy this property as the following example shows:

\begin{example}
Consider the pairs $(\bA^2_{\bF_3}, \Div(xy) )$ and $(\bA^2_{\bF_3}, \Div(x^2 - y^2) )$.  Choose $\epsilon = {1 \over 3}$.  Then I claim that the pair $(\bA^2_{\bF_3}, {2 \over 3} \Div(xy) + {1 \over 3} \Div(x^2 - y^2))$ is not sharply $F$-pure at the origin $\bm = (x, y)$.  To see this, using Fedder's criterion for pairs, see \cite[Theorem 4.1]{SchwedeSharpTestElements} (compare with \cite{FedderFPureRat} and \cite{TakagiInversion}), it is enough to observe that
\[
\begin{split}
(xy)^{\lceil {(2)(3^e - 1) \over 3}  \rceil} (x^2 - y^2)^{\lceil {(3^e - 1) \over 3} \rceil} = (xy)^{(2)3^{e-1}} (x^2 - y^2)^{3^{(e-1)}} =\\
(x^{(4)3^{(e-1)}} y^{(2)3^{(e-1)}} - x^{(2)3^{(e-1)}} y^{(4)3^{(e-1)}})   \subseteq \bm^{[3^e]}
\end{split}
\]
for all $e > 0$.  Of course, if one uses ``$F$-purity'', see \cite{HaraWatanabeFRegFPure} and \cite{TakagiInversion}, instead of ``sharp $F$-purity'', this sort of example does not occur.
\end{example}

Of course, we'd also like to know whether  log Calabi-Yau varieties over a field of characteristic zero have dense sharply F-split type. As we showed in Section 5, this is essentially equivalent to the well-known (difficult) conjecture that log canonical singularities have dense sharply F-pure type, in the special case of standard graded rings.


\providecommand{\bysame}{\leavevmode\hbox to3em{\hrulefill}\thinspace}
\providecommand{\MR}{\relax\ifhmode\unskip\space\fi MR}
\providecommand{\MRhref}[2]{%
  \href{http://www.ams.org/mathscinet-getitem?mr=#1}{#2}
}
\providecommand{\href}[2]{#2}

\end{document}